\documentclass[11pt,a4paper]{article}
\usepackage[]{geometry}
\usepackage[mathcal]{euscript}
\usepackage{bm,url}                     
\usepackage{amsfonts}
\usepackage{amssymb}
\usepackage{amsmath}
\usepackage{amsthm}
\usepackage[utf8x]{inputenc}
\usepackage{float}
\usepackage[pdftex]{graphicx}
\DeclareGraphicsExtensions{.bmp,.png,.pdf,.jpg,.ps}
\usepackage{colortbl}  
\usepackage{todonotes}             
\usepackage{booktabs}
\usepackage{amsmath,amsfonts,amssymb,amsthm,booktabs,color,epsfig,graphicx,hyperref,url}

\newtheorem{thm}{Theorem}
\newtheorem{lem}[thm]{Lemma}
\newtheorem{corollary}{Corollary}
\newtheorem{Proposition}{Proposition}
\newtheorem{remark}{Remark}

\usepackage[english]{babel}
\usepackage{amsmath}
\usepackage{amsfonts}
\usepackage{amssymb}
\usepackage{graphicx}
\usepackage{color}
\usepackage{array}
\usepackage{mathrsfs}
\usepackage{hyperref}

\definecolor{violet}{rgb}{0.7,0,0.6}
\definecolor{OliveGreen}{RGB}{85,107,47}

\title{An Independence Test Based on Recurrence Rates}
\author{Juan Kalemkerian \\
Universidad de la República, Facultad de Ciencias\\ \\
Diego Fern\'andez\\
Universidad de la República, Facultad de Ciencias Económicas \\ y Administración}

\begin{document}
\maketitle
\begin{abstract}
A new test of independence between random elements is presented in this
article. The test is based on a functional of the Cram\'{e}r-von Mises type, which is
applied to a $U$-process that is defined from the recurrence rates. Theorems of
asymptotic distribution under $H_{0},$ and consistency under a wide class of
alternatives are obtained. The results under contiguous alternatives are also
shown. The test has a very good behaviour under several alternatives, which
shows that in many cases there is clearly larger power when compared to other tests that are 
widely used in literature. 
In addition, the new test  could be used for discrete or continuous time series.
\end{abstract}


\newpage

\section{Introduction}

Let $\left( X_{1},Y_{1}\right) ,\left( X_{2},Y_{2}\right) ,...,\left(
X_{n},Y_{n}\right) $ i.i.d. sample of $\left( X,Y\right) ,$ $X\in S_{X}$ and 
$Y\in S_{Y}$, where $S_{X}$ and $S_{Y}$ are metric spaces. When we have the
following hypothesis test: $H_{0}:$ $X$ and $Y$ are independent random
elements, we are under the so called independent tests. The independence
tests have been developed in the first instance for the $S_{X}=S_{Y}=\mathbb{R}$ case, based on the pioneering work of Galton 
\cite{Galton} and Pearson
\cite{Pearson} (this is the famous correlation test, which is widely used today). The limitations of
this hypothesis test are well known and they have motivated several different proposals
in this topic, such as the classical rank test (e.g. Spearman,\cite{Spearman}, Kendall, 
\cite{Kendall} or Blomqvist, \cite{Blomqvist}). Another classic and intuitive result 
can be found in Hoeffding \cite{Hoeffding2}, where the test statistic is defined by
$\int \int \left ( F_{X,Y}(x,y)-F_X(x)F_Y(y) \right )^{2}dF_{X,Y}(x,y)$, although it is not widely used.
Independence between random
vectors is addressed for the first time in Wilks \cite{Wilks}. Genest and R\'emillard \cite{Genest} propose a test based on copulas for continuous
random variables. Kojadinovic and Holmes \cite{Kojadinovic}, generalize this result for random vectors using a Cram\'er-von Mises
type statistic. Bilodeau and Lafaye de Micheaux \cite{Bilodeau}, propose a test of independence between random vectors,
each of which has a normal marginal distribution. Continuing in some sense this work, Beran et al. \cite{Beran} propose a
universally consistent test for random vectors, from empirical
multidimensional distributions. Gretton et al. \cite{Gretton} propose a universally
consistent test based on Hilbert-Schmidt norms. Another consistent test is
proposed by Sz\'{e}kely et al. \cite{Szekely,Szekely2}, which defines the concept of distance
covariance. This test has its origin in \cite{Bakirov} and it has since become very popular. It has been used
and has had a considerable impact from the moment that it was proposed. More
recently, Heller et al. \cite{Heller} propose a test that in many cases has much
more powerfull than the distance covariance test. In
his monograph, Boglioni \cite{Boglioni}
compares several alternatives of these tests by means of intense work of power calculations. Because the tests
proposed in Beran et al. \cite{Beran} and Heller et al. \cite{Heller} have very good performance under
several alternatives, in Section 4 we will compare them with the test that we
propose in our work.

Starting from another point of view, Eckman et al. \cite{Eckman} introduce the
recurrence plot (RP). This is a very important graphical \ tool
to understand the dynamics of a time series in high dimension. Eckman et al.'s \cite{Eckman} 
generated an appreciable amount of work and is currently 
applied in many different areas in which mathematical models are used, whether
probabilistic or deterministic. The RP is a graphical tool that shows
the recurrence in a time series $\left( X\right) $ and it is constructed using the
recurrence matrix $RM\left( X\right) $ as defined by $RM_{ij}\left( X\right) =%
\mathbf{1}_{\left\{ \left\Vert X_{i}-X_{j}\right\Vert <r\right\} }$, where $r$
is an appropriate parameter. The objective of this tool is to determine
the patterns in a time series. The choice of $r$ is a key point to detect
patterns and several suggestions have been made on how to appropriately find it.
Marwan \cite{Marwan} gives a historical review
of recurrence plots techniques, together with everything developed from them.
However, the potential of these techniques has not \ yet been studied in depth from
the point of view of  mathematical statistics.

The main objective of this article is to propose a hypothesis test to detect
dependence between two random elements, $X$ and $Y$, based on recurrence
rates by using the information of $\mathbf{1}_{\left\{ d \left ( 
X_{i},X_{j} \right  )<r\right\} }$ and $\mathbf{1}_{\left\{ d \left ( Y_i ,Y_j \right )
<s\right\} }$ for \ any values of $r$ and $s.$  One 
advantage of our test is that instead of choosing appropriate values of $r$ and $s$,
we use the information generated by both samples for all of the possible values of
$r$ and $s$. In our
test, $X$ and $Y$ can take values in any metric space. Therefore, our test can be
used to test if $X$ and $Y$ are independent in the case where $X$ and $Y$
are random variables, random vectors or time series. We can then replace
the norms by distances. 

The rest of this paper is organized as follows. In Section 2, we
give the definitions of recurrence rates for $X,$ for $Y$ and for joint $%
\left( X,Y\right) $ and we propose the statistical procedure to make the
decision between $H_{0}$ vs $H_{1}.$ The statistics are based on a functional
of the Cram\'{e}r von-Mises type applied to a $U$-process defined from the
recurrence rates of $X,$ $Y$ and $\left( X,Y\right).$ We also  give
the theoretical results, which are the asymptotic distribution and consistency
of the test statistic (Subsection 2.1), and the behavior under contiguous
alternatives (Subsection 2.2). In Section 3, we describe how the
test can be implemented, including a formula to obtain the
statistic \ for the test. In Section 4, we use simulations to show the
performance of the test against others by power comparison in the cases
where $X$ and $Y$ are random variables or random vectors. We also compute power in the case
where $X$ and $Y$ are  discrete and continuous time series. Like Heller et al.'s \cite{Heller} test,
 our test is based on distances between the elements of the sample. Likewise, our
 test had very good performance under several alternatives. 
Our concluding remarks are given in Section 5. Appendix gives the proofs of the results that are established in Section 2.

\section{Test approach and theoretical results}

Given $\left( X_{1},Y_{1}\right) ,\left( X_{2},Y_{2}\right) ,...,\left(
X_{n},Y_{n}\right) $ i.i.d. sample of $\left( X,Y\right) $ where $X\in S_X,$ $Y\in S_Y$ 
where $S_X$ and $S_Y$ are metric spaces, and given $r,s>0.$ To simplify the notation and without 
risk of confusion, we will use the same letter $d$ for the distance function in both
metric spaces $S_X$ and $S_Y$. 

We define the recurrence rate for the sample of $X$ and $Y$ as%
\[
RR_{n}^{X}(r):=\frac{1}{n^{2}-n}\sum_{i\neq j}\mathbf{1}_{\left\{ d \left ( 
X_{i},X_{j} \right  )<r\right\} } 
\]%
\[
RR_{n}^{Y}(s):=\frac{1}{n^{2}-n}\sum_{i\neq j}\mathbf{1}_{\left\{ d \left ( Y_i ,Y_j \right ) <s\right\} } 
\]%
respectively, and the joint recurrence rate for $\left( X,Y\right) $ as
\[RR_{ n}^{X,Y}(r,s):=\frac{1}{n^{2}-n}\sum_{i\neq j}\mathbf{1}_{\left\{
d \left ( X_{i},X_{j} \right ) <r\text{ },\text{ }d \left ( Y_i ,Y_j \right ) <s \right\} }.\]

We define $p_{X}(r):=P\left( d \left ( 
X_{1},X_{2} \right  )<r\right) $ the probability that the distance
between any two elements of the sample $X$ is less than $r.$ Similarly, we
define the probability between 
three points as $p_{X}^{\left( 3\right) }(r):=P\left( d \left ( X_{1},X_{2} \right )
<r,\text{ } d \left ( X_{1},X_{3} \right ) <r \right) $
and analogously $p_{Y}$ and $p_{Y}^{\left( 3\right) }.$ 

We also  need to define $p_{X,Y}(r,s):=P\left( d \left ( 
X_{1},X_{2} \right  )<r,\text{ }d \left ( Y_{1},Y_{2} \right )
<s \right) .$

The strong law of large numbers for $U$-statistics  (\cite{Hoeffding}) allows us to affirm that for any $r,s>0$, \begin{equation}\label{cs}
               RR_{n}^{X}(r)\overset{a.s.}{\rightarrow }p_{X}(r),  \ \ RR_{n}^{Y}(s)%
\overset{a.s.}{\rightarrow }p_{Y}(s) \ \ \text{and} \ \ RR_{n}^{X,Y}(r,s)\overset{a.s.}{%
\rightarrow }p_{X,Y}(r,s).
             \end{equation}

\bigskip\ We want to test $H_{0}:$ $X$ and $Y$ are independent, against $%
H_{1}:$ $H_{0}$ does not hold.

If $H_{0}$ is true, then $p_{X,Y}(r,s)=p_{X}(r)p_{Y}(s)$ for all $r,s>0$, and we expect
that if $n$ is large,  $RR_{n}^{X,Y}(r,s)\cong RR_{n}^{X}(r)RR_{n}^{Y}(s)$ for
any $r, s >0.$
Then, we propose to build the test statistic, to work with the process $\{E_n(r,s)\}_{r,s>0}$ where 
\begin{equation}\label{En}
 E_n(r,s):=\sqrt{n} \left (RR_{n}^{X,Y}(r,s)- RR_{n}^{X}(r)RR_{n}^{Y}(s) \right ).
\end{equation}

Therefore, it is natural to reject $H_{0}$ when $T_n >c$ where 
\begin{equation}\label{cvm}
T_n:= n\int_0^{+ \infty} \int_0^{+ \infty}\left (
RR_{n}^{X,Y}(r,s)-RR_{n}^{X}(r)RR_{n}^{Y}(s)\right )^{2}dG(r,s)  
 \end{equation}%
where $c$ is a constant and $G$ is a distribution function.
  
Throughout this work, we use the notation $\phi$ and $\varphi$ for distribution and density function of $N(0,1)$ random variable respectively, and for each $m$, the set \[I_{m}^{n}:=\left\{ \left(
i_{1},...,i_{m}\right) :\text{ }i_{j}\neq i_{k}\text{ for all }j\neq k\text{%
	, and }i_{j}\in \left\{ 1,...,n\right\} \text{ for all }%
j=1,...,m\right\} .\]

Now we will formulate the asymptotic results of our test statistic. First, we will show a result that guarantees 
the asymptotic distribution of $T_n$ under $H_0$. We will also present a result that establishes a consistency of our test
under a wide class of alternatives.
Second, we will analyze the asymptotic bias when we consider contiguous alternatives.

\subsection{Asymptotic results under $H_0$ and consistency}

We start with the next lemma, in which we obtain the formula for the asymptotic autocovariance function of the process
$\{E_n(r,s)\}_{r,s>0}$ under $H_0$.

\begin{lem}
Given $r,r',s,s'>0,$ and $\left( X_{1},Y_{1}\right) ,\left( X_{2},Y_{2}\right) ,...,\left(
X_{n},Y_{n}\right) $ i.i.d. in $S_X \times S_Y$ 
where $X$ and $Y$ are independent, then  
\begin{equation*}
 \lim_{n\rightarrow +\infty }\mathbb{COV}\left( E_n(r,s),E_n(r',s') \right )
 =
  \end{equation*}
 \begin{equation}\label{covarianzas}
 4\left( p_{X}^{\left( 3\right) }(r \wedge r')-p_{X}(r)p_X(r')\right)
\left( p_{Y}^{\left( 3\right) }(s \wedge s')-p_{Y}(s)p_Y(s')\right).
\end{equation}

\end{lem}
The following lemma will be useful to reduce asymptotic convergence of the process
$\{E_n(r,s)\}_{r,s>0}$ to the convergence of an approximate $U-$ process that we
will call $\{E_n^{\prime }(r,s)\}_{r,s>0}$ and is defined as follows
\[E_{n}^{\prime }(r,s):=\frac{\sqrt{n}}{n(n-1)(n-2)(n-3)} \times \]
\begin{equation}\label{eprima}
\sum_{(i,j,k,h)\in I_{4}^{n}}\left( \mathbf{%
1}_{\left\{ d \left ( X_{i},X_{j} \right ) <r,\text{ }d \left ( Y_i ,Y_j \right ) <s\right\} }-\mathbf{1}_{\left\{ d \left ( 
X_{i},X_{j} \right  )<r,\text{ } d \left ( Y_{h},Y_{k} \right )
<s\right\} }\right).
\end{equation}

\begin{lem}
	Given $\left( X_{1},Y_{1}\right) ,\left( X_{2},Y_{2}\right) ,...,\left(
	X_{n},Y_{n}\right) $ i.i.d. in $S_X \times S_Y$ , then
 \begin{equation*}
E_n(r,s)=\sqrt{n}\left( RR_{n}^{X,Y}(r,s)-RR_{n}^{X}(r)RR_{n}^{Y}(s)\right)
=E_{n}^{\prime }(r,s)-H_n(r,s)
\end{equation*}%
where 
\begin{equation*}
0\leq H_n\left( r,s\right) \leq \frac{4}{\sqrt{n}}\text{ for all }r,s>0.
\end{equation*}
\end{lem}

To obtain the weak convergence of the process $\{E_n(r,s)-\mathbb{E}(E_n(r,s))\}_{r,s>0}$
to a centered Gaussian process  (therefore the asymptotic distribution
of the statistics $T_n$ defined in (\ref{cvm}) is determined), we
 will use Theorem 4.10 obtained by Arcones \& Gin\'{e} \cite{Arcones}:

Let $\left( S,\mathcal{S},P\right) $ be a probability space, and for all $i\in 
\mathbb{N}$, $X_{i}:S\rightarrow S$ are i.i.d. sequence with $\mathcal{L}\left(
X_{i}\right) =P.$ Given $m$, let $\mathcal{F}$ be a class of measurable functions on $%
S^{m},$ the $U$-process based on $P$ and indexed by $\mathcal{F}$ is 
\begin{equation*}
U_{m}^{n}\left( f\right) =\frac{\left( n-m\right) !}{m!}\sum_{\left(
i_{1},...,i_{m}\right) \in I_{m}^{n}}f\left(
X_{i_{1},}X_{i_{2}},...,X_{i_{m}}\right) 
\end{equation*}%
 where $f\in \mathcal{F}$.

Given $\varepsilon >0$, assume that exists $
 \mathcal{L} = \left\{ l_{1},l_{2},...,l_{v}\right\}$, 
 $\mathcal{U} = \left\{ u_{1},u_{2},...,u_{\upsilon }\right\}$ such that
$\mathcal{L},\mathcal{U}\subset L^{2}$ and for all 
\begin{equation}\label{entropia}
 f\in \mathcal{F}, 
\text{exists } l_{f} \in \mathcal{L}\text{ and }u_{f}\in \mathcal{U}\text{
where } l_{f}\leq f\leq u_{f}\text{ }a.s.\text{ and }\mathbb{E}\left(
u_{f}-l_{f}\right) ^{2}<\varepsilon ^{2}.
\end{equation}

\begin{equation}
N_{\left[ { \ }\right] }^{\left( 2\right) }\left( \varepsilon ,\mathcal{F},P^{m}\right)
=\min \left\{ \upsilon : (\ref{entropia}) \text{ holds} \right\} .
\end{equation}

\textbf{Theorem (Arcones $\&$ Gin\'{e} 1993)
} \[\]

\textit{If
\begin{equation}\label{condicionArconesGine}
\int_{0}^{+\infty }\left( \log N_{\left[ {\ }\right] }^{\left( 2\right)
}\left( \varepsilon ,\mathcal{F},P^{m}\right) \right) ^{1/2}d\varepsilon <+\infty 
\end{equation}%
then 
\begin{equation}\label{Arcones}
\mathcal{L}\left( \sqrt{n}\left( U_{m}^{n}-P^{m}\right) f\right) \overset{w}{%
\rightarrow }\mathcal{L}\left( mG_{p}\circ P^{m-1}f\right) \text{ in }%
l^{\infty }\left( \mathcal{F}\right) 
\end{equation}%
 where $G_{P}$ is the Brownian bridge associated with $P.$
}

Convergence in the space $l^{\infty }\left( \mathcal{F}\right)$, is in the sense of Hoffmann-J\o rgensen, see (\cite{Gine-Zinn}).

\begin{thm}\label{asymptotic under H0}
Given  $\left( X_{1},Y_{1}\right) ,\left( X_{2},Y_{2}\right) ,...,\left(
X_{n},Y_{n}\right) $ i.i.d. in $S_X \times S_Y.$ If the distribution functions of $d(X_1,X_2)$ and $d(Y_1,Y_2)$ are continuous, then
 \begin{equation}
  \{E_n(r,s)-\mathbb{E}(E_n(r,s))\}_{r,s>0} \xrightarrow{w} \{E(r,s)\}_{r,s>0}
 \end{equation}
 where $\{E(r,s)\}_{r,s>0}$ is a centered Gaussian process. 

\end{thm}

\begin{remark}
Observe that our process $\{E_{n}(r,s)\}_{r,s>0}$ lies in $L^{2}(dG)$ (because $G$ is a probability measure). Therefore, 
our test statistic $T_n$ is $\left ||\{E(r,s)\}_{r,s>0} \right || $,  thus, the functional is continuous.
 
\end{remark}

\begin{remark}
 Given $r,s>0$ and $\left( X_{1},Y_{1}\right) ,\left( X_{2},Y_{2}\right) ,...,\left(
X_{n},Y_{n}\right) \in \mathbb{R}^2
 $ i.i.d. sample of $\left( X,Y\right) $ where the
marginals $X,Y$ are $N\left( 0,1\right) $ independent. Then \[ \sqrt{n}\left(
RR_{n}^{X,Y)}(r,s)-RR_{n}^{}(r)RR_{n}^{Y}(s)\right) \overset{w}{\rightarrow }N\left(
0,\sigma _{X,Y}^{2}(r,s)\right) \] where 
\begin{equation*}
\sigma _{X,Y}^{2}(r,s)=
4\left( \int_{-\infty }^{+\infty }\left( \phi \left( x+r\right) -\phi
\left( x-r\right) \right) ^{2}\varphi \left( x\right) dx-\left( 2\phi
\left( r/\sqrt{2}\right) -1\right) ^{2}\right) \times  \end{equation*}

\begin{equation}\label{sigma2normal}
        \left( \int_{-\infty }^{+\infty }\left( \phi \left( x+s\right) -\phi
\left( x-s\right) \right) ^{2}\varphi \left( x\right) dx-\left( 2\phi
\left( s/\sqrt{2}\right) -1\right) ^{2}\right) .
       \end{equation}
\end{remark}

If $d(X_1,X_2)$ and $d(Y_1,Y_2)$ are not independent, then our test is consistent.

\begin{thm}\label{consistent theorem}
 Given $\left( X_{1},Y_{1}\right) ,\left( X_{2},Y_{2}\right) ,...,\left(
 X_{n},Y_{n}\right) $ i.i.d. in $S_X \times S_Y$. If $dG(r,s)=g(r,s)drds$,  $g(r,s)>0$ for all $r,s>0,$ and $d \left ( 
X_{1},X_{2} \right  )$, $d \left ( Y_{1},Y_{2} \right ) $ are
continuous and not independent random variables, then $T_{n}\overset{P}{%
\rightarrow }+\infty $ as $n\rightarrow +\infty .$
\end{thm}
 The next corollary follows  from Theorem \ref{consistent theorem}.

\begin{corollary}\label{normal consistency}

If $\left( X,Y\right) \sim N\left( 0,\Sigma \right) $, where $X$ and $Y$ are not
independent, and $dG(r,s)=g(r,s)drds$,  $g(r,s)>0$ for all $r,s>0,$ then $T_{n}\overset{P}{\rightarrow }+\infty $ as $n\rightarrow
+\infty $.
\end{corollary}
\begin{remark}
 Consider $\left( X_{1},Y_{1}\right) $, $\left( X_{2},Y_{2}\right) $ in $%
\mathbb{R}^{2}$ i.i.d. with joint density $f_{X,Y}$ and joint distribution $F$ such that $\left\vert
X_{1}-X_{2}\right\vert $ and $\left\vert Y_{1}-Y_{2}\right\vert $ are
independent. 

Then 
$$
\alpha \left( r,s\right) :=P\left( \left\vert X_{1}-X_{2}\right\vert \leq r,%
\text{ }\left\vert Y_{1}-Y_{2}\right\vert \leq s\right) = $$ $$\iint_{\mathbb{R}%
^{2}}f_{X,Y}(x_{1},y_{1})dx_{1}dy_{1}\int_{x_{1}-r}^{x_{1}+r}dx_{2}%
\int_{y_{1}-s}^{y_{1}+s}f_{X,Y}(x_{2},y_{2})dy_{2}=
$$
\begin{equation*}
\iint_{\mathbb{R}^{2}}P\left( x_{1}-r\leq X_{1}\leq x_{1}+r,\text{ }%
y_{1}-s\leq Y_{2}\leq y_{1}+s\right) f_{X,Y}(x_{1},y_{1})dx_{1}dy_{1}=
\end{equation*}%
\begin{equation*}
\mathbb{E}\left( F\left( X+r,Y+s\right) -F\left( X+r,Y-s\right)
-F\left( X-r,Y+s\right) +F\left( X-r,Y-s\right) \right) .
\end{equation*}%
Similarly, 
$$\beta \left( r,s\right) :=P\left( \left\vert X_{1}-X_{2}\right\vert \leq
r\right) P\left( \text{ }\left\vert Y_{1}-Y_{2}\right\vert \leq s\right) =$$
$$
\mathbb{E}\left( F_{X}\left( X+r\right) -F_{X}\left( X-r\right) \right) 
\mathbb{E}\left( F_{Y}\left( Y+r\right) -F_{Y}\left( Y-r\right) \right) .
$$
Then, $\alpha \left( r,s\right) =\beta \left( r,s\right) $ for all $r,s>0.$

Of course, it could  happen that condition $\alpha \left( r,s\right) =\beta
\left( r,s\right) $ for all $r,s>0$ is fulfilled, and nevertheless $X$ and $Y$
are not independent. This is the restricted type of distributions  that do
not satisfy the conditions of our consistency theorem.
\end{remark}

\subsection{Contiguous alternatives }

In this subsection we will analyze the behavior of this test under contiguous
alternatives.

More explicitly, given $\left( X_{1},Y_{1}\right) ,\left( X_{2},Y_{2}\right)
,...,\left( X_{n},Y_{n}\right) $ i.i.d. in $\mathbb{R}^{p}\times \mathbb{R}%
^{q}$, consider
\[H_{0}: f_{X,Y}(x,y)=f_{X}(x)f_{Y}(y) \ \ \ \text{for all} \left(
x,y\right) \] (i.e. $X$ and $Y$ are independent), vs \[H_{n}:%
f_{X,Y}(x,y)=f_{X,Y}^{\left( n\right) }(x,y) \ \ \ \text{ for all} \left( x,y\right) \]
where  $f_{X,Y}^{(n)}(x,y)=c_{n}\left( \delta \right) f_{X}(x)f_{Y}(y)\left(
1+\frac{\delta }{2\sqrt{n}}k_{n}(x,y)\right) ^{2},$ $\delta >0,$ $%
c_{n}\left( \delta \right) $ is a constant such that $%
f_{X,Y}^{(n)}(x,y)$ be a density, and the functions $k_{n}$ verify the
 conditions (i) and (ii) that are given below:

Define $L_{0}^{2}=L^{2}\left( dF_{0}\right) $ for $%
dF_{0}(x,y)=f_{X}(x)f_{Y}(y)dxdy$, the distribution function of $\left(
X,Y\right) $ under $H_{0},$ analogously define $L_{0}^{1}.$
\begin{itemize}
 \item[(i)] Exists a function $K\in L_{0}^{1}$ such that $k_{n}\leq K$ for all $n$
 \item[(ii)] Exists $k\in L_{0}^{2}$ such that $k_{n}\overset{L_{0}^{2}}{\rightarrow }k$, $%
\left\Vert k\right\Vert =1.$

\end{itemize}

It can be proven that conditions (i) and (ii) imply contiguity (Caba\~{n}a \cite{henry}).

The $\delta $ coefficient is introduced so that $\left\Vert k\right\Vert =1.$
The function $\delta k$ is called asymptotic drift.

We will show in the following lines that under $H_{n},$ the process $\left\{
E_{n}(r,s)\right\} _{r,s>0}$ has the same asymptotic limit as under $H_{0}$
plus a deterministic drift.

We use the notation $\mathbb{E}^{\left( n\right) }\left( T\right) $ and $%
P^{\left( n\right) }\left( \left( X,Y\right) \in A\right) $ for the
expectation value of $T$, and the probability of the set $\left\{ \left(
X,Y\right) \in A\right\} $ under $H_{n}$ respectively. Analogously we use $%
\mathbb{E}^{\left( 0\right) }\left( T\right) $ and $P^{\left( 0\right)
}\left( \left( X,Y\right) \in A\right) $ under $H_{0}.$
\begin{Proposition}
\[\]
 Under $H_n$ 
 \begin{equation*}
\mathbb{E}^{\left( n\right) }\left( E_{n}(r,s)\right) \rightarrow \delta \mu (r,s) 
\text{ \ as }n\rightarrow +\infty \ \text{ for all } r,s>0.
\end{equation*}
where $ \mu (r,s)=  $ \begin{equation}
      \iiiint_{A_{r,s}}\left( k(x_{1},y_{1})+k(x_{2},y_{2})\right)
f_{X}(x_{1})f_{Y}(y_{1})f_{X}(x_{2})f_{Y}(y_{2})dx_{1}dx_{2}dy_{1}dy_{2},%
      \end{equation}
and  $A_{r,s}:=\left\{ \left( x_{1},y_{1},x_{2},y_{2}\right) \in 
\mathbb{R}^{2p+2q}:\text{ } d( x_{1},x_{2})  <r,\text{ }%
d( y_{1},y_{2}) <s\right\}.$
\end{Proposition}

With a little more work, using the Le Cam third lemma (Le Cam \&  Yang,
\cite{Le Cam} and Oosterhoff \& Van Zwet, \cite{Oosterhoff}) it is possible to prove that under $%
H_{n}$, 
\begin{equation*}
\left\{ E_{n}(r,s)\right\} _{r,s>0}\overset{w}{\rightarrow }\left\{
E(r,s)+\delta \mu \left( r,s\right) \right\} _{r,s>0}
\end{equation*}%
where $\left\{ E(r,s)\right\} _{r,s>0}$ is the limit process under $H_{0}$
and $\mu \left( r,s\right) =$
\begin{equation*}
\iiiint_{A_{r,s}}\left(
k(x_{1},y_{1})+k(x_{2},y_{2})\right)
f_{X}(x_{1})f_{Y}(y_{1})f_{X}(x_{2})f_{Y}(y_{2})dx_{1}dx_{2}dy_{1}dy_{2}.
\end{equation*}

Therefore, under $H_{n}$ 
\begin{equation*}
T_{n}\overset{w}{\rightarrow }\int_{0}^{+\infty }\int_{0}^{+\infty }\left(
E(r,s)+\delta \mu \left( r,s\right) \right) ^{2}dG(r,s).
\end{equation*}

\section{Implementation of the test}

\subsection{$X$ and $Y$ are random variables}
In the case where $X$ and $Y$ are continuous random variables, we observe that $X$ and $Y$ are independent;
it is equivalent to say that $%
X^{\prime }=\phi ^{-1}\left( F_{X}\left( X\right) \right) $ and $Y^{\prime
}=\phi ^{-1}\left( F_{Y}\left( Y\right) \right) $ are independent, where $%
F_{X}$ and $F_{Y}$ are the distribution functions of $X$ and $Y$,
respectively. If we apply the test procedure to $X^{\prime }$ and $Y^{\prime }$,
then we have the advantage that now the variables are on the same scale and each
has a normal centered distribution that approximates  to the
hypotheses of Remark 2. In addition, in this case the formula (\ref{sigma2normal}) for $\sigma _{X^{\prime },Y^{\prime }}^{2}(r,s)$ is completely determined. Another additional advantage is that under $H_{0}$ (%
$X^{\prime }$ and $Y^{\prime }$ are independent and $N\left( 0,1\right) $),
for small values of $n$, we can calculate the critical values at $5\%$ or
another level because we will know the distribution of $T_{n}$ under $H_{0}.$
Where $X$ and $Y$ are random vectors, the same transformation can be applied in each coordinate.
\noindent To give an idea of the variability of the process $\{E_n(r,s)\}_{r,s>0}$, in Figure \ref{gra_sigma2} we show the values of $\sigma _{X^{\prime },Y^{\prime }}^{2}(r,r)$
for different values of $r.$
The maximum is $0.06409$ and is reached in $r=1.3488$.

   \begin{figure}[h!]
 \centering
   \includegraphics[scale=0.6]{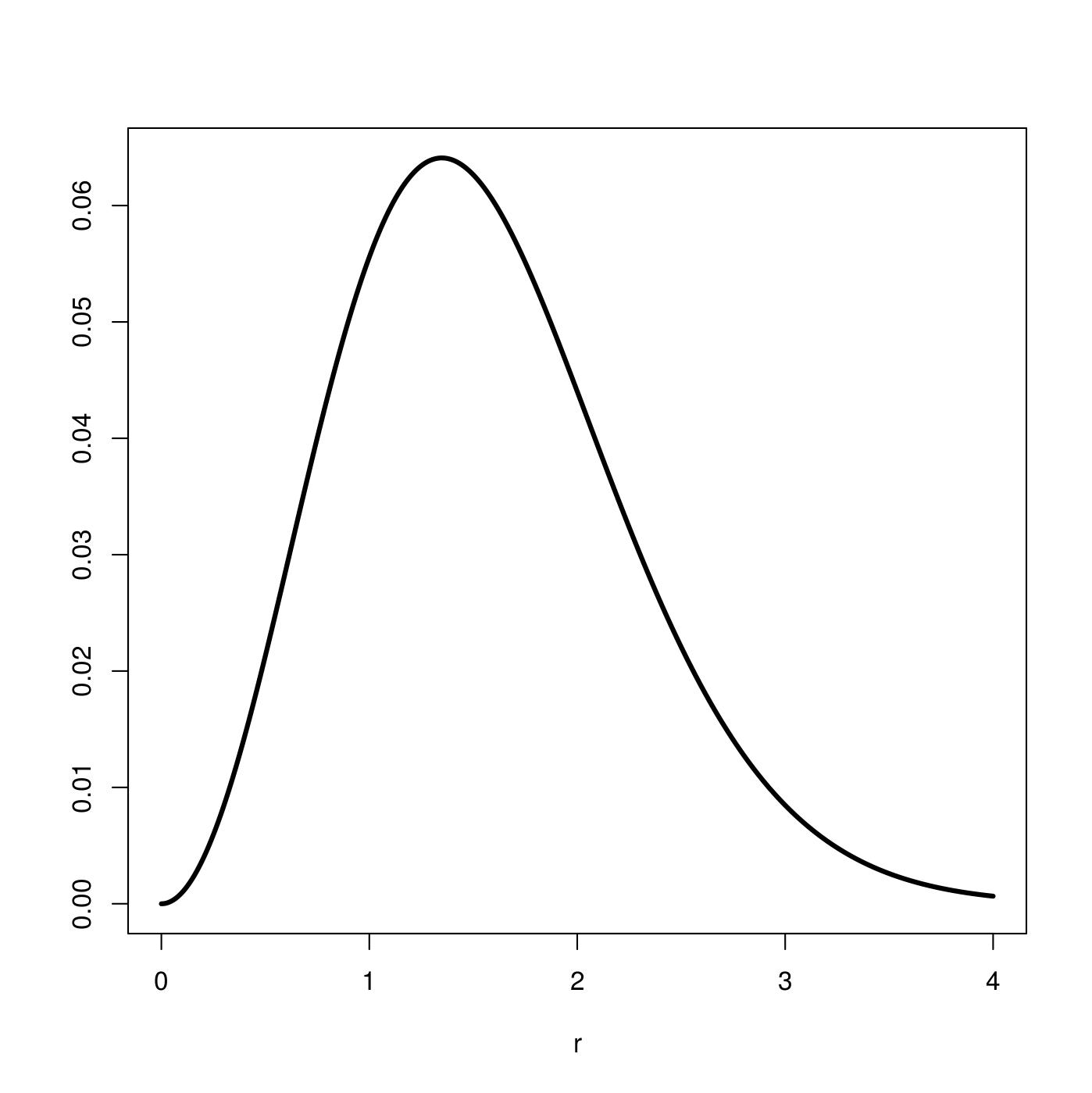} 

   \caption{$\sigma^{2}_{X,Y}(r,r)$ in function of $r$, in the case $X$ and $Y$ are 
   independent and $N(0,1)$.} 
  \label{gra_sigma2}
\end{figure}

\subsection{General case}

As happens in many statistical applications, we are able to have a moderately
small sample size. However, an
erroneous decision can be made if the researcher uses the p-value (or the critical value) obtained through the
asymptotic distribution to make the decision in the hypothesis test. Therefore, when we have a sample of size $n$,
it is preferable to estimate the p-value (or the critical value) by estimating the distribution of
the $T_{n}$ for this value of $n.$
Moreover, in our test, the asymptotic distribution \ is difficult to obtain
because we need to conduct several simulations of a centered continuous Gaussian
processes indexed in $D=\left( 0,+\infty \right) \times \left( 0,+\infty
\right)$. We then need to calculate the integral in $D.$ 

To calculate the p-value or the critical value of the test for fixed $n$
we can proceed as explained in the following lines.
Fixed $n$, if $H_{0}$ is true, we do not know the distribution of $T_{n},$
but given the observed value from our sample that we call $t_{obs}$, we
could generate, by a permutation procedure,  a large sample of $T_{n}$ with
which we can estimate $P\left( T_{n}\geq t_{obs}\right) $. 
Given $\left( X_{1},Y_{1}\right) ,\left( X_{2},Y_{2}\right) ,...,\left(
X_{n},Y_{n}\right) $ i.i.d. sample of $\left( X,Y\right).$ Observe that the
distribution of $T_{n}$ depends of the joint distribution of $\left(
X_{1},Y_{1}\right) ,\left( X_{2},Y_{2}\right) ,...,\left( X_{n},Y_{n}\right)
.$  If $H_{0}$ is true, and if we consider any $\sigma :\left\{
1,2,3,...,n\right\} \rightarrow \left\{ 1,2,3,...,n\right\} $  permutation
of the index set, then the joint distribution of $\left( X_{1},Y_{1}\right)
,\left( X_{2},Y_{2}\right) ,...,\left( X_{n},Y_{n}\right) $ and the joint
distribution of $\left( X_{\sigma (1)},Y_{1}\right) ,\left( X_{\sigma
(2)},Y_{2}\right) ,...,\left( X_{\sigma (n)},Y_{n}\right) $ are the same.
 Consider $S\left( n\right) =\left\{ \sigma _{1},\sigma _{2},...,\sigma
_{n!}\right\} $ the set of all the permutation $\sigma :\left\{
1,2,...,n\right\} \rightarrow \left\{ 1,2,...,n\right\}.$ Suppose that the
sample $\left( X_{1},Y_{1}\right) ,...,\left( X_{n},Y_{n}\right) $ is fixed and
consider  $Z$ defined by  $Z=T_{n}\left( \left(
X_{\sigma _{i}(1)},Y_{1}\right) ,...,\left( X_{\sigma _{i}(n)},Y_{n}\right)
\right) $ with probability $1/n!$ for each $i=1,2,...,n!.$ If we take $Z_1,Z_2,...,Z_m$ i.i.d. sample of $Z$, we can
estimate the value of $p_n=P\left( T_{n}\geq t_{obs}\right) $
simply by using $\widehat{p}_{n}^{\left( m\right) }=\frac{1}{m}\sum_{i=1}^{m}\mathbf{1}_{\left\{ Z_{i}\geq
t_{obs}\right\} }$ for $m$ large enough. Define the random variables $B_{i}=\sum_{j=1}^{m}\mathbf{1}_{\left\{ Z_{j}=T_{n}\left( \left(
X_{\sigma _{i}(1)},Y_{1}\right) ,...,\left( X_{\sigma _{i}(n)},Y_{n}\right)
\right) \right\} }.$ for $i=1,2,,...,n!$. Observe that $B_{i}$ is distributed
as  Bin$\left( m,1/n!\right) $ for each $i=1,2,...,n!.$ Then 
\[
\widehat{p}_{n}^{\left( m\right) }=\frac{1}{m}\sum_{j=1}^{m}\mathbf{1}%
_{\left\{ Z_{j}\geq t_{obs}\right\} }=\frac{1}{m}\sum_{i=1}^{n!}B_{i}\mathbf{%
1}_{\left\{ T_{n}\left( \left( X_{\sigma _{i}(1)},Y_{1}\right) ,...,\left(
X_{\sigma _{i}(n)},Y_{n}\right) \right) \geq t_{obs}\right\} }
\]%
converges as $m\rightarrow +\infty $ to $\frac{1}{n!}\sum_{i=1}^{n!}\mathbf{1%
}_{\left\{ T_{n}\left( \left( X_{\sigma _{i}(1)},Y_{1}\right) ,...,\left(
X_{\sigma _{i}(n)},Y_{n}\right) \right) \geq t_{obs}\right\} }$ a.s.
If we now
consider that $\left( X_{1},Y_{1}\right) ,...,\left( X_{n},Y_{n}\right) $
are random elements that can take an expected value, and  we obtain (using dominated convergence) $\mathbb{E}\left( 
\widehat{p}_{n}^{\left( m\right) }\right) \underset{m\rightarrow +\infty }{%
\rightarrow }p_{n}$, then $\widehat{p}_{n}^{\left( m\right) }$ is an
asymptotically unbiased estimator of $p_{n}.$

\subsection{A simple method to choose the weight function}
The performance of our test depends on the choice of the weight function. The weight function
can be chosen by the researcher in each particular case. According to Theorem \ref{consistent theorem},
 we can use any function $G$ such that $dG(r,s)=g(r,s)drds$ where $g(r,s)>0$ for any $r,s>0$. It would be interesting to
 study some kind of optimality in the choice of the $G$ function, under certain kind of alternatives.
Consequently, we propose a simple method to chose the $G$ function. As  will be seen in the next section, this simple choice of $G$, has very good performance 
 under the alternatives studied in this work.
 
Define $dG(r,s)=g_1(r)g_2(s)drds$, where $g_1$ and $g_2$ are Gaussian densities.
In the case of $g_1$ we can use $\mu_1=\mathbb{E}\left ( d(X_1,X_2)\right ) $ and $\sigma_1^{2} =\mathbb{V}\left ( d(X_1,X_2)\right )$.
The values of $\mu_1$ and $\sigma_1$ can easily be estimated by the sample $d \left ( X_{i},X_{j} \right ) $
with $(i,j) \in I_2^n$. We can proceed similarly with the election of $\mu_2$ and $\sigma_2$ for the density $g_2$. In this way, we
give more weight in the neighbourhoods of the average distance between two independent observations $X_1$ and $X_2$ for $g_1$, and analogously for 
$g_2$.
Meanwhile, observe that we can avoid the problem of choosing $G$, if we use $T'_n=\sqrt{n}\sup_{r,s>0} \left | RR_n^{X,Y}(r,s)
-RR_n^{X}(r)RR_n^{Y}(s) \right | $ to test independence because all of the theoretical results obtained in this work for $T_n$ are still
valid for $T'_n$.

\subsection{Computing the statistic}
In this subsection we will see how to calculate the statistic $T_n$.
We will consider the case in which $dG(r,s)=g_1(r)g_2(s)drds$ where $g_1$ and $g_2$ are density functions
with $G_1$ and $G_2$ their respective distribution functions.

 \begin{equation*}
 \int_{0}^{+\infty }\int_{0}^{+\infty }\left( RR_{n}^{X,Y}\left( r,s\right)
 -RR_{n}^{X}\left( r\right) RR_{n}^{Y}\left( s\right) \right) ^{2}g_1
 \left( r\right) g_2 \left( s\right) drds=
 \end{equation*}%
 \begin{equation*}
 \int_{0}^{+\infty }\int_{0}^{+\infty }\left[ RR_{n}^{X,Y}\left( r,s\right) %
 \right] ^{2}g_1 \left( r\right) g_2 \left( s\right)
 drds+
  \end{equation*}
   \begin{equation*}
 \int_{0}^{+\infty }\left[ RR_{n}^{X}\left( r\right) \right] ^{2}g_1
 \left( r\right) dr\int_{0}^{+\infty }\left[ RR_{n}^{Y}\left( s\right) \right]
 ^{2}g_2 \left( s\right) ds
 \end{equation*}
 \begin{equation}\label{a+b-2c}
  -2\int_{0}^{+\infty }\int_{0}^{+\infty
 }RR_{n}^{X,Y}\left( r,s\right) RR_{n}^{X}\left( r\right) RR_{n}^{Y}\left(
 s\right) g_1 \left( r\right) g_2 \left( s\right) drds:=A_n+B_n-2C_n.
 \end{equation}

To simplify the notation and for the rest of this section, we will call
 $N=n(n-1).$ We will also index  $d \left ( X_{i},X_{j} \right ) $
with $(i,j) \in I_2^n$ in the form $Z_{1},Z_{2},...,Z_{N}$. 
Analogously, we use the same indexes as $Z's$, $T_{1},T_{2},...,T_{N}$ to the values 
$d \left ( Y_{i},Y_{j} \right )$. We will also call $Z_{1}^{\ast
},Z_{2}^{\ast },...,Z_{N}^{\ast }$ to the order statistics of $Z's$, and analogously
 $T_{1}^{\ast },T_{2}^{\ast },...,T_{N}^{\ast }$.

\begin{equation*}
\int_{0}^{+\infty }\left[ RR_{n}^{X}\left( r\right) \right] ^{2}g_1
\left( r\right) dr=\frac{1}{N^{2}}\sum_{i\neq j}\sum_{k\neq
	h}\int_{0}^{+\infty }\mathbf{1}_{\left\{ d \left ( X_{i},X_{j} \right )
	<r,\text{ }d \left ( X_{h},X_{k} \right ) <r\right\} }g_1 \left(
r\right) dr=
\end{equation*}%
\begin{equation*}
\frac{1}{N^{2}}\sum_{i=1}^{N}\sum_{j=1}^{N}\int_{0}^{+\infty }\mathbf{1}%
_{\left\{ Z_{i}<r,\text{ }Z_{j}<r\right\} }g_1 \left( r\right) dr=\frac{1%
}{N^{2}}\sum_{i=1}^{N}\sum_{j=1}^{N}\left( 1-G_1 \left( \max \left\{
Z_{i},Z_{j}\right\} \right) \right)=
\end{equation*}
\begin{equation*}
\frac{1}{N^{2}}\sum_{i=1}^{N}\sum_{j=1}^{N}\left( 1-G_1 \left( \max \left\{
Z_{i}^{\ast },Z_{j}^{\ast }\right\} \right) \right) =1-\frac{1}{N^{2}}%
\sum_{i=1}^{N}\left( 2\sum_{j=1}^{i-1}G_1 \left( Z_{i}^{\ast }\right) +G_1
\left( Z_{i}^{\ast }\right) \right) =
\end{equation*}%
\begin{equation*}
1-\frac{1}{N^{2}}\sum_{i=1}^{N}\left( 2\left( i-1\right) G_1 \left(
Z_{i}^{\ast }\right) +G_1 \left( Z_{i}^{\ast }\right) \right) =1-\frac{1}{%
	N^{2}}\sum_{i=1}^{N}\left( 2i-1\right) G_1 \left( Z_{i}^{\ast }\right) .
\end{equation*}%

Analogously
\begin{equation*}
\int_{0}^{+\infty }\left[ RR_{n}^{Y}\left( s\right) \right] ^{2}g_2
\left( s\right) ds=1-\frac{1}{N^{2}}\sum_{i=1}^{N}\left( 2i-1\right) G_2
\left( T_{i}^{\ast }\right) .
\end{equation*}%
Then 
\begin{equation}\label{b}
B_n=\left (1-\frac{1}{N^{2}}\sum_{i=1}^{N}\left( 2i-1\right) G_1
\left( Z_{i}^{\ast }\right) . \right )\left (1-\frac{1}{N^{2}}\sum_{i=1}^{N}\left( 2i-1\right) G_2
\left( T_{i}^{\ast }\right) . \right )
\end{equation}
\begin{equation*}
A_n= \int_{0}^{+\infty }\int_{0}^{+\infty }\left[ RR_{n}^{X,Y}\left( r,s\right) %
\right] ^{2}g_1 \left( r\right) g_2 \left( s\right) drds=
\end{equation*}
\begin{equation*}
\frac{1}{%
	N^{2}}\sum_{i=1}^{N}\sum_{j=1}^{N}\int_{0}^{+\infty }\mathbf{1}_{\left\{
	Z_{i}<r,\text{ }Z_{j}<r\right\} }g_1 \left( r\right) dr\int_{0}^{+\infty
}\mathbf{1}_{\left\{ T_{i}<s,\text{ }T_{j}<s\right\} }g_2 \left(
s\right) ds=
\end{equation*}
\begin{equation}\label{a}
\frac{1}{N^{2}}\sum_{i=1}^{N}\sum_{j=1}^{N}\left( 1-G_1 \left( \max \left\{
Z_{i},Z_{j}\right\} \right) \right) \left( 1-G_2 \left( \max \left\{
T_{i},T_{j}\right\} \right) \right) .
\end{equation}

\begin{equation*}
C_n=\int_{0}^{+\infty }\int_{0}^{+\infty }RR_{n}^{X,Y}\left( r,s\right)
RR_{n}^{X}\left( r\right) RR_{n}^{Y}\left( s\right) g_1 \left( r\right)
g_2 \left( s\right) drds=
\end{equation*}

\begin{equation*}
\frac{1}{N^{3}}\sum_{i\neq j}\sum_{k\neq h}\sum_{l\neq m}\int_{0}^{+\infty
}\int_{0}^{+\infty }\mathbf{1}_{\left\{ d \left ( X_{i},X_{j} \right ) <r,%
\text{ }d \left ( Y_{i},Y_{j} \right ) <s,\text{ }d \left ( X_{h},X_{k} \right ) <r,
\text{ }d \left ( Y_{l},Y_{m} \right )
<s\right\} }g_1 \left( r\right) g_2 \left( s\right) drds=
\end{equation*}%
\begin{equation}\label{c}
\frac{1}{N^{3}}\sum_{i=1}^{N}\sum_{j=1}^{N}\sum_{k=1}^{N}\left( 1-G_1
\left( \max \left\{ Z_{i},Z_{j}\right\} \right) \right) \left( 1-G_2 \left(
\max \left\{ T_{i},T_{k}\right\} \right) \right) .
\end{equation}
Then

\begin{equation}\label{Tn}
	T_n=n(A_n+B_n-2C_n)
\end{equation}
where $A_n$, $B_n$ and $C_n$ are given in the formulas (\ref{a}), (\ref{b}) and (\ref{c}) respectively.

\section{A simulation study}

In this section we will compare the performance of our test with respect to other recently 
proposed tests that have good performance. Tables 1 to 6 show the power of our test for different functions $G$
and also for other tests,  for  $n=30$, $n=50$ and $n=80$ sample sizes.
All power calculations that we have considered have been calculated at the significance level of $5
\%$. The calculations were
made using (\ref{Tn}) and taking as a function of weights $dG(r,s)=g_1(r)g_2(s)drds$ where $g_1=g_2=g$ is the density function of 
a $N(\mu, \sigma^{2})$
random variable for some values of $\mu$ and $\sigma^{2}$, except for the last column, where we take the functions $g_1$ and $g_2$ suggested 
in Subsection 3.3.
We will compare the power of our test with respect to the test
proposed in Heller et al. \cite{Heller} (which we will call HHG), the test 
of covariance distance proposed in Sz\'ekely et al. \cite{Szekely} (which we will call DCOV) and the test proposed 
in Gretton et al. \cite{Gretton} (which we will call HSIC).
In Subsection 4.1 we will consider the case in which $X$ and $Y$ are random variables; that is,
$(X,Y) \in \mathbb{R}^{2}$. Meanwhile, in Subsection 4.2 we consider examples in dimensions greater than two. 
Lastly, in Subsection 4.3 we simulate discrete and continuous time series for certain alternatives and representspower as a
function of sample size. In this case, we take the functions $g_1$ and $g_2$ suggested 
in Subsection 3.3.
\subsection{$X$ and $Y$ are random variables}
Table 3 considers Heller et al.'s \cite{Heller} tests, which are called ``Parabola'', ``Two
parabolas'', ``Circle'', ``Diamond'', ``W-shape'' and ``Four independent clouds'' and which are defined as follows:\\
Parabola: $X\sim U\left( -1,1\right) ,$ $Y=\left( X^{2}+U\left( 0,1\right)
\right) /2.$\\
Two parabolas:  $X\sim U\left( -1,1\right) ,$ $Y=\left( X^{2}+U\left(
0,1\right) /2\right) $ with probability $1/2$ and $Y=-\left( X^{2}+U\left(
0,1\right) /2\right) $ with probability $1/2.$\\
Circle: $U\sim U\left( -1,1\right) $, $X=\sin \left( \pi U\right) +N\left(
0,1\right) /8$, $Y=\cos \left( \pi U\right) +N\left( 0,1\right) /8.$\\
Diamond: $U_{1},U_{2}\sim U\left( -1,1\right) $ independent, $X=\sin \left(
\theta \right) U_{1}+\cos \left( \theta \right) U_{2},$ $Y=-\sin \left(
\theta \right) U_{1}+\cos \left( \theta \right) U_{2}$ for $\theta =\pi /4.$\\
W-shape: $U \sim U(-1,1)$, $U_1, U_2 \sim U(0,1)$ independent. $X = U+U_1/3$ and  $Y=4\left (U^2-1/2 \right )^2+U_2/n.$\\
Four independent clouds: $X=1+Z_1/3$ with probability $1/2$, $X=-1+Z_2/3$ with probability $1/2$ and 
$Y=1+Z_3/3$ with probability $1/2$, $Y=-1+Z_4/3$ with probability $1/2$, where $Z_1, Z_2, Z_3, Z_4 \sim N(0,1)$ are independent. \\
Observe that in ``Four independent clouds'', $H_0$ is true, and the power in all the cases should be around $0.05$.
 In all cases, the critical values of our test were calculated  through 50000 replications and the power
of all of the tests considered from 10000 replications.
The first three columns of Table~\ref{t1} give the power of the HHG, DCV and HSIC tests. Column 4
gives the maximum power among the classic correlation test: Pearson, Spearman and Kendall, which we call PSK.
Columns 5, 6 and 7 give the power of our test for different $g=g_1=g_2$ function considered in the weight
function $G$. In column 8, we use the function $g_1$ and $g_2$ proposed in Subsection 3.3, analogously in Table 2 and Table 3.
Figure 2  give us $n=1000$ simulations of the alternatives considered in this subsection.

\begin{figure}[ht]
\begin{center}
\begin{tabular}{ll}
\includegraphics[scale=0.35]{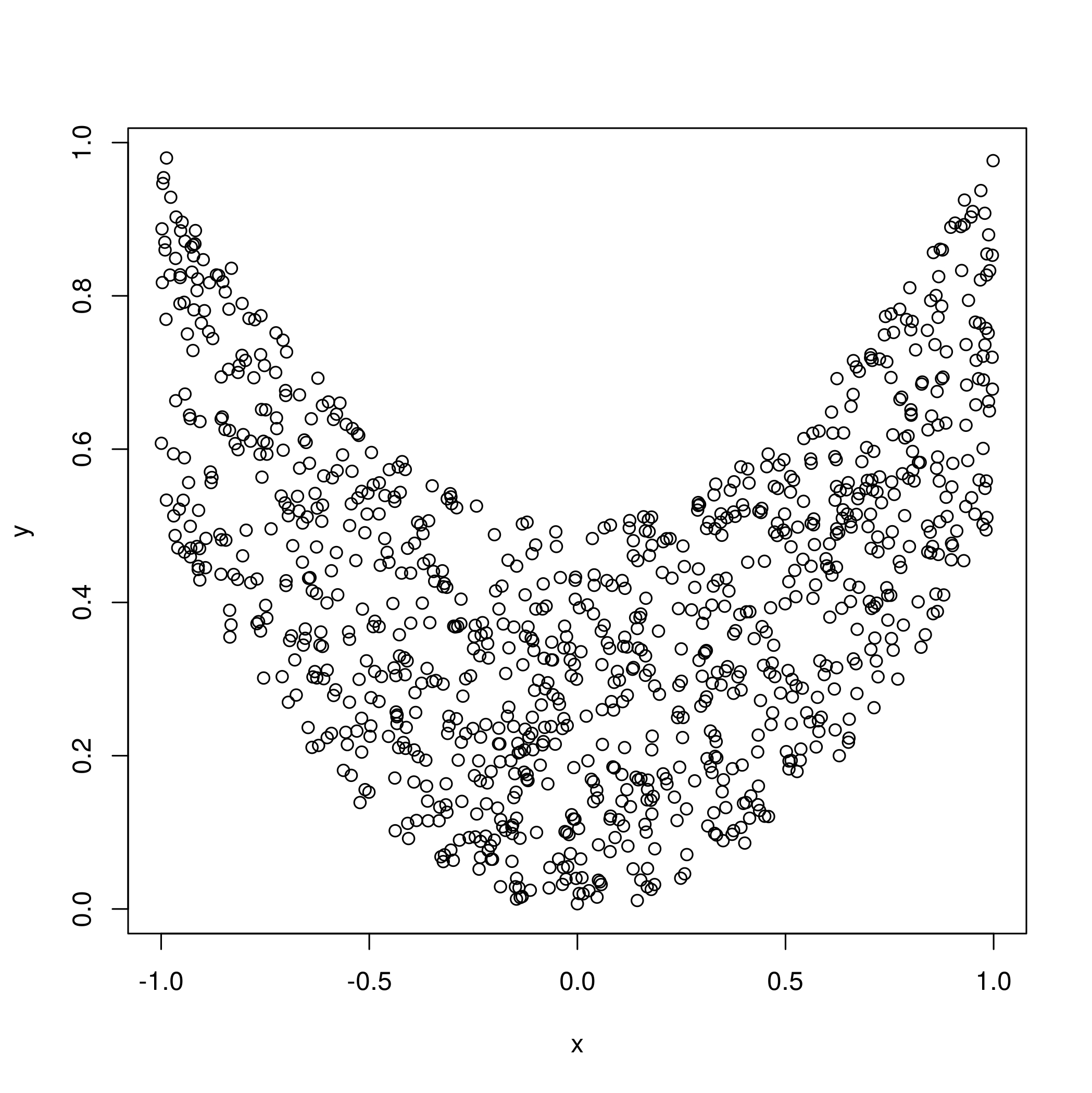}
& \includegraphics[scale=0.35]{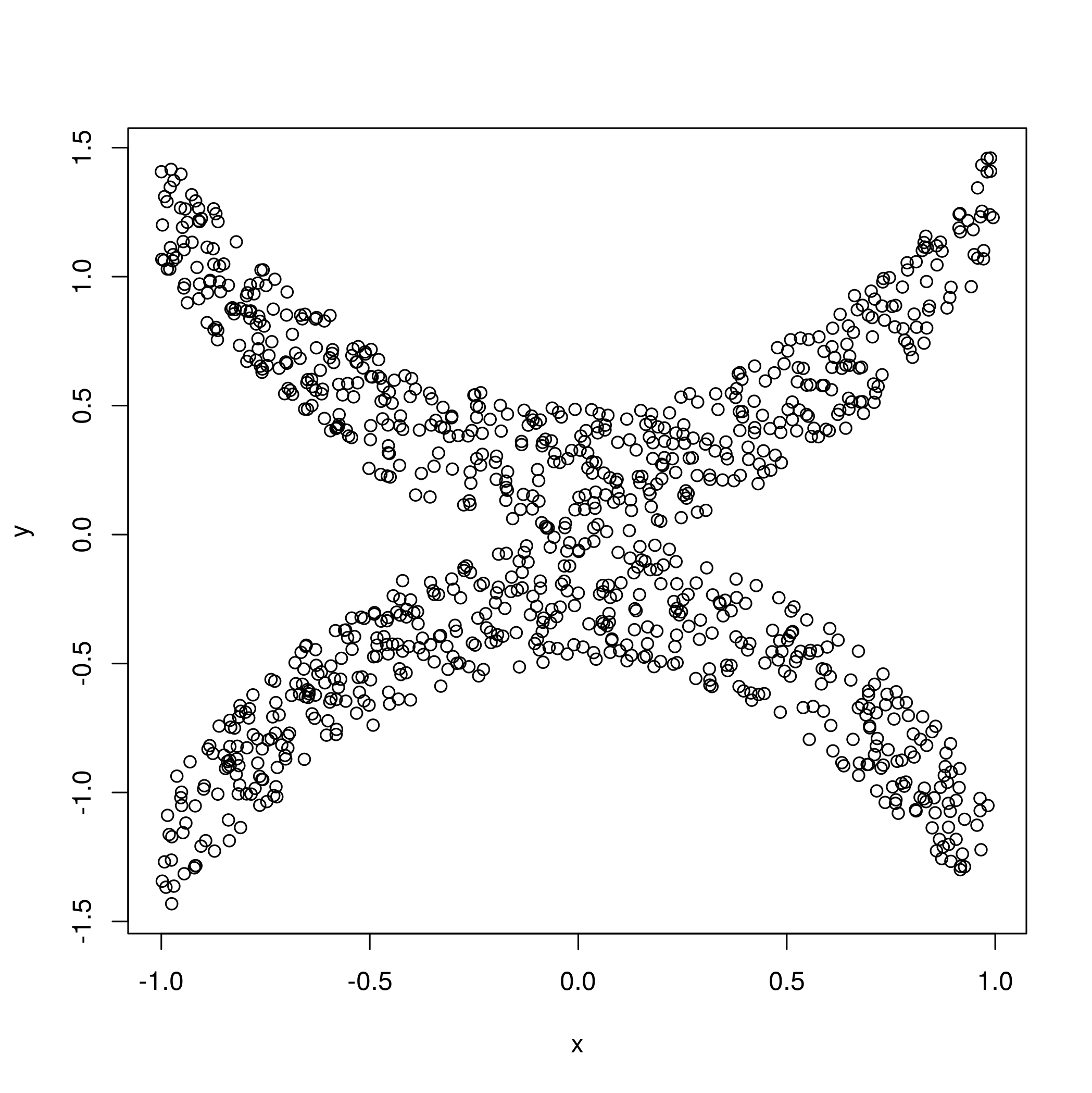}\\
 \includegraphics[scale=0.35]{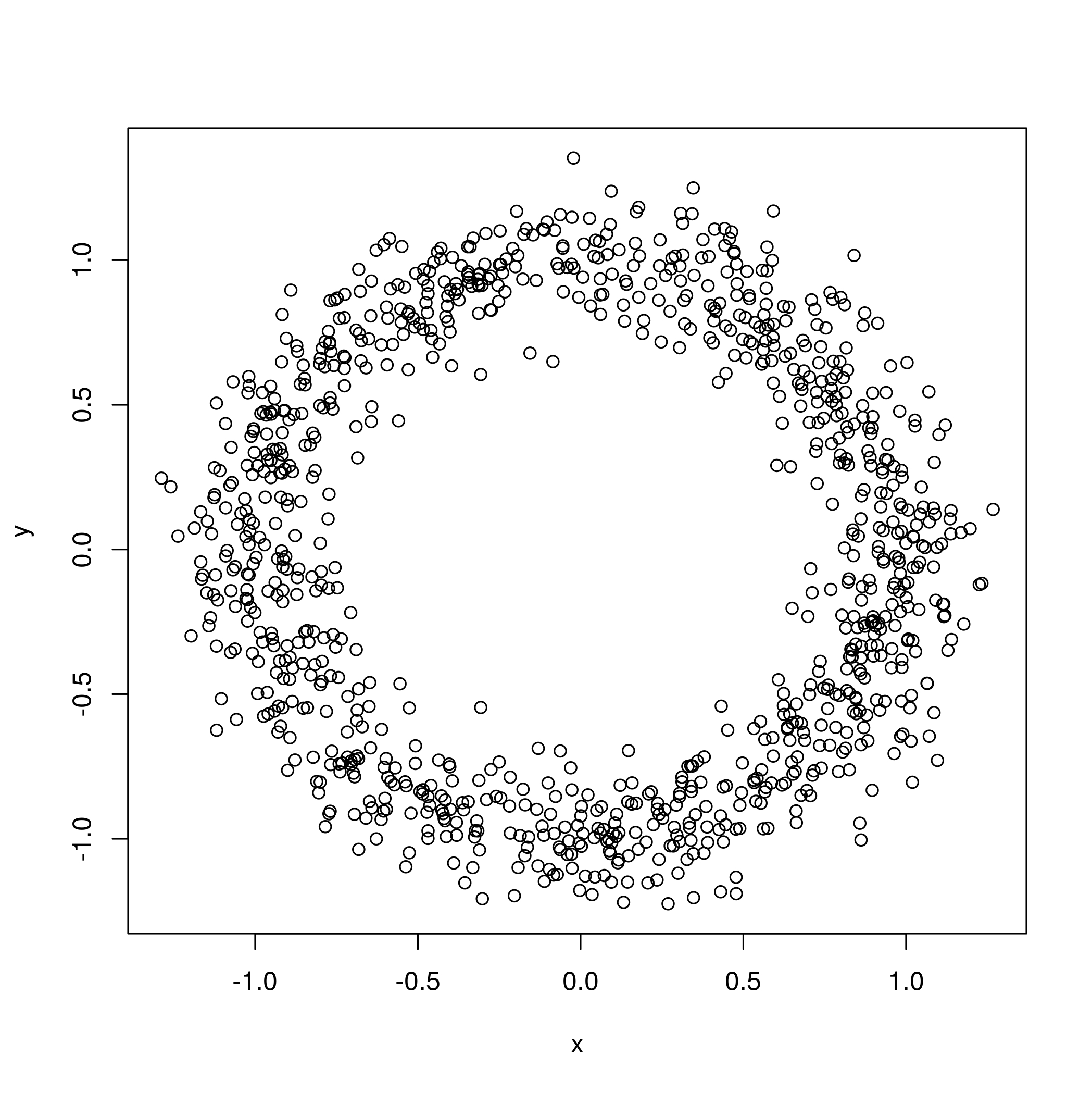} &
\includegraphics[scale=0.35]{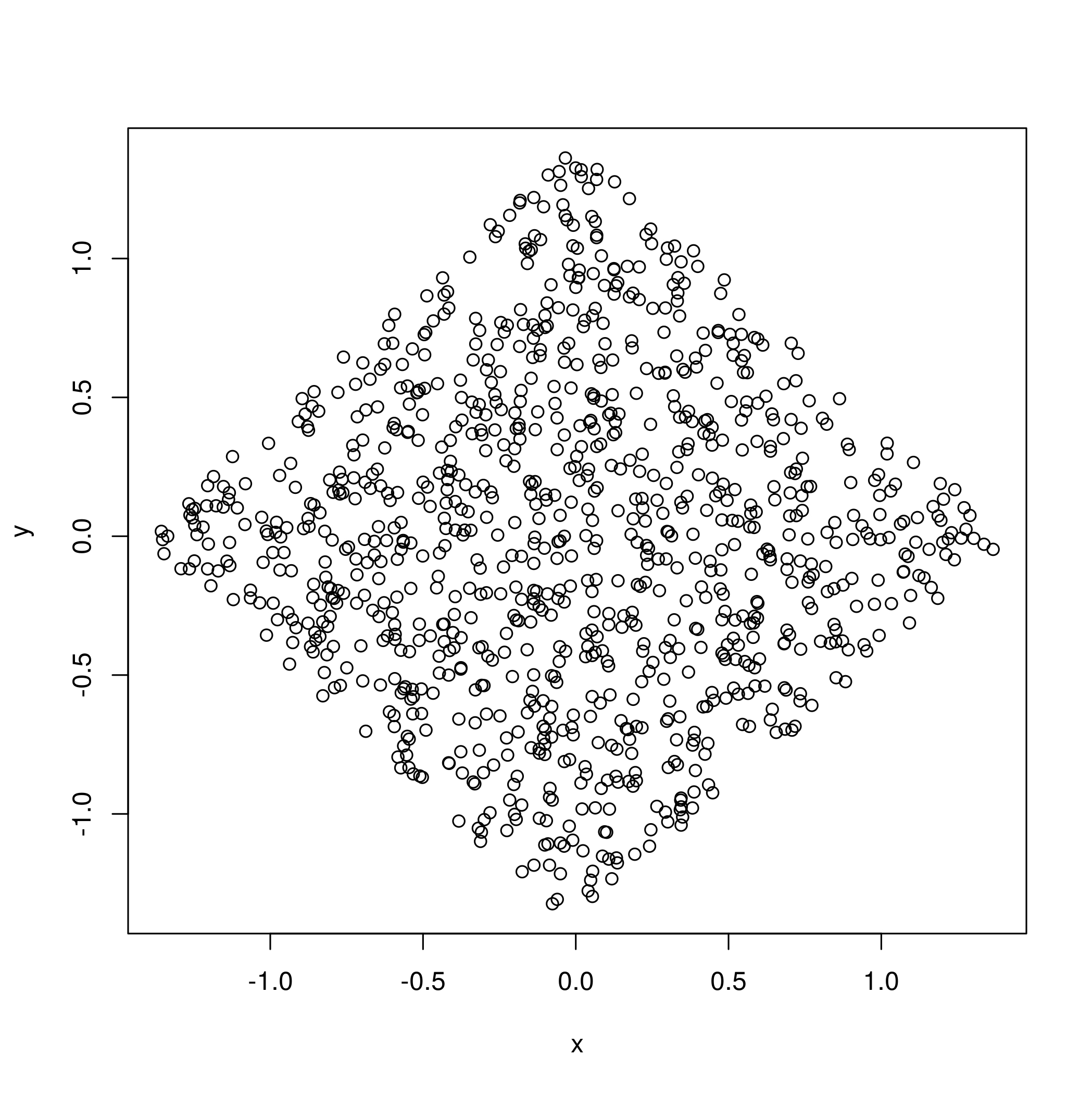}\\
 \includegraphics[scale=0.35]{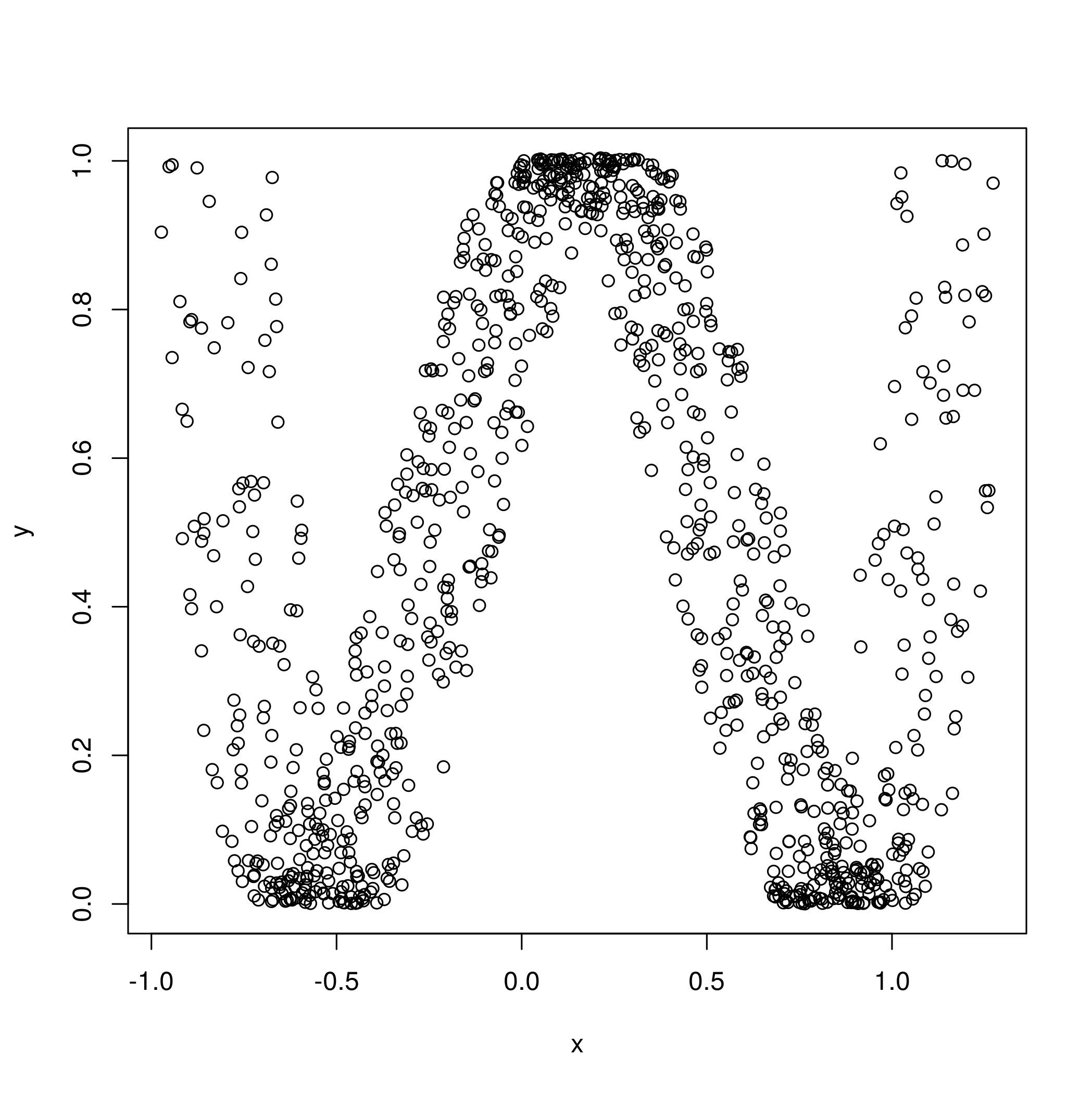}
& \includegraphics[scale=0.35]{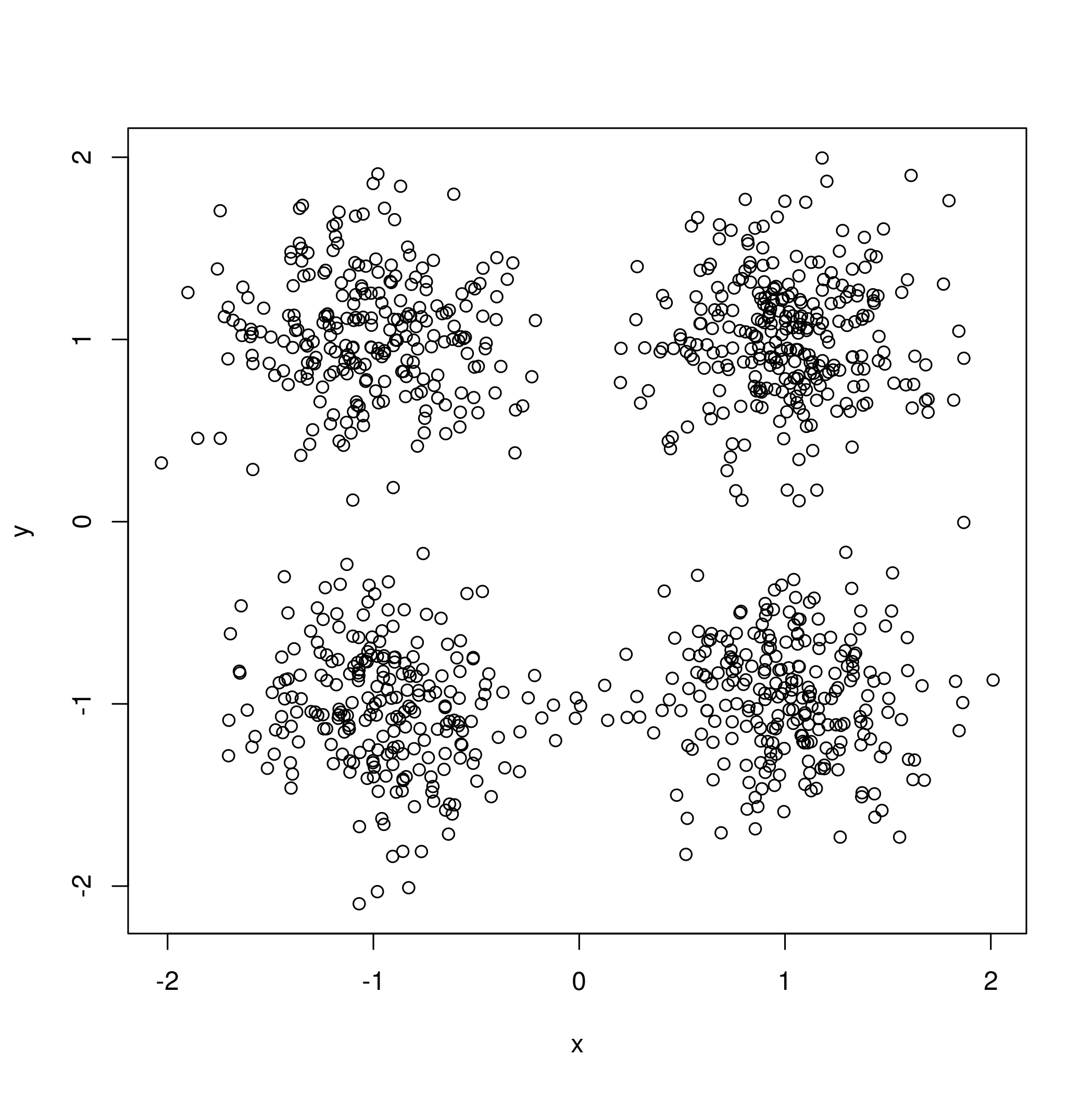}\\
\end{tabular}
\end{center}
\caption{Parabola, Two parabolas, circle, diamond, wshape and four independent clouds.}\label{ML:figuras262728}
\end{figure}

\begin{table*}[ht]
\caption {Power comparison for the different test for sample size of $n=30$.}
\label{t1}

\centering
\begin{tabular}{|r|rrrrrrrr|}

    \hline
   Test & HHG  & DCOV & HSIC & PSK & N(1,1) & N(0,1) & N(1,4) & $g_1,g_2$ \\
   \hline 
   Parabola & 0.791 & 0.522 & 0.733 & 0.103 & 0.824 & \textbf{0.831} & 0.814 & 0.817 \\
   \hline
   2 parabolas & 0.962 & 0.204 & 0.849 & 0.194 & \textbf{1.000 }& \textbf{1.000 }& \textbf{1.000} & \textbf{1.000} \\
   \hline 
   Circle & 0.646 & 0.051 & 0.488 & 0.096 & 0.923 & 0.716 & \textbf{0.947} & 0.823 \\
   \hline
   Diamond & 0.283 & 0.030 & 0.262 & 0.016 & 0.422 & 0.139 & \textbf{0.477} & 0.395 \\
   \hline
   W-shape & \textbf{0.908} & 0.569 & 0.856 & 0.179 & 0.788 & 0.887 & 0.782 & 0.874 \\
   \hline \hline
   4 clouds & 0.052 & 0.053 & 0.053 & 0.046 & 0.052 & 0.052 & 0.051 & 0.051 \\
   \hline
   
\end{tabular}             
\end{table*}

\begin{table*}[ht]\caption{ Power comparison for the different test for sample size 
of $n=50$.}
\label{t2}
\centering
\begin{tabular}{|r|rrrrrrrr|}

    \hline
   Test & HHG  & DCOV & HSIC & PSK & N(1,1) & N(0,1) & N(1,4) & $g_1,g_2$ \\
   \hline 
   Parabola & 0.983 & 0.854 & 0.957 & 0.114 & 0.979 & 0.983 & \textbf{1.000} & 0.975 \\ 
   \hline
   2 parabolas & \textbf{1.000} & 0.354 & 0.997 & 0.198 & \textbf{1.000} &\textbf{ 1.000} & \textbf{1.000} &\textbf{ 1.000 }\\
   \hline
   Circle & 0.985 & 0.075 & 0.914 & 0.008 & 0.999 & 0.997 & \textbf{1.000} & 0.995 \\
   \hline
   Diamond & 0.664 & 0.048 & 0.545 & 0.013 & 0.836 & 0.630 & \textbf{0.884} & 0.761 \\
   \hline
   W-shape & \textbf{0.999} & 0.935 & 0.988 & 0.077 & 0.989 & 0.998 & 0.987 & 0.979 \\
   \hline \hline
   4 clouds & 0.050 & 0.047 & 0.048 & 0.046 & 0.512 & 0.055 & 0.054 & 0.051 \\
   \hline
   
\end{tabular}             
\end{table*}

\begin{table*}[ht]\caption{ Power comparison for the different test for sample size of $n=80$.}
\label{t3}

\centering
\begin{tabular}{|r|rrrrrrrr|}

    \hline
   Test & HHG  & DCOV & HSIC & PSK & N(1,1) & N(0,1) & N(1,4) & $g_1,g_2$ \\
   \hline 
   Parabola & \textbf{1.000} & 0.994 & \textbf{1.000} & 0.105 & \textbf{1.000} & \textbf{1.000 }& \textbf{1.000} & \textbf{1.000} \\
   \hline
   2 parabolas & \textbf{1.000} & 0.700 & \textbf{1.000} & 0.201 &\textbf{ 1.000} & \textbf{1.000 }& \textbf{1.000} & \textbf{1.000 }\\
   \hline 
   Circle & \textbf{1.000} & 0.196 & 0.999 & 0.004 &0.999 & \textbf{1.000} & \textbf{1.000}  & \textbf{1.000} \\
   \hline
   Diamond & 0.948 & 0.096 & 0.853 & 0.003 & 0.836 & 0.953 & \textbf{1.000} & 0.999 \\
   \hline
   W-shape & \textbf{1.000 }& 0.999 & \textbf{1.000 }& 0.085 & 0.988 & \textbf{1.000 }& \textbf{1.000} & \textbf{1.000} \\
   \hline \hline
   4 clouds & 0.047 & 0.047 & 0.047 & 0.049 & 0.051 & 0.049 & 0.055 & 0.057 \\
   \hline
   
\end{tabular}             
\end{table*}

\subsection{$X$ and $Y$ are random vectors}

In our test, the distance considered for the calculations of recurrences measures is given
for the Euclidean norm. Because the Euclidean distance increases with the dimension,
the densities of $N(0,4)$ and $N(2,4)$ were aggregated in the  columns 6 and 7. In this subsection,
we consider the last two alternatives in Table 3, and in Table 4 of Heller et al. \cite{Heller}, which 
we will call ``Logarithmic'', ``Epsilon'' and ``Quadratic'' tests and which are defined as follows:\\
Logarithmic: $X,Y\in \mathbb{R}^{5}$ where $X_{i}$ $\sim$ $N\left(
0,1\right) $ are independent, $Y_{i}=\log \left( X_{i}^{2}\right) $ for $%
i=1,2,3,4,5.$\\
Epsilon: $X,Y,\varepsilon \in \mathbb{R}^{5}$ where $X_{i},\varepsilon
_{i}$ $\sim$ $N\left( 0,1\right) $ are independent, $Y_{i}=\varepsilon _{i}X_{i}$
for $i=1,2,3,4,5.$\\
Quadratic: $X,Y,\varepsilon \in \mathbb{R}^{5}$ where $X_{i},\varepsilon
_{i}$ are independent, $X_{i}\sim N\left( 0,1\right) ,$ $\varepsilon
_{i}\sim N\left( 0,3\right) ,$ $Y_{i}=X_{i}+4X_{i}^{2}+\varepsilon _{i}$  $%
i=1,2,$ $Y_{i}=\varepsilon _{i}$ for all $i=3,4,5.$\\
We also add the alternatives considered in Boglioni, which are called ``2D-pairwise independent'' and are defined as follows:\\
2D-pairwise independent: $X,Z_{0},Y_{1}\sim N\left( 0,1\right) $
independent, $Y=\left( Y_{1},Y_{2}\right) $ where $Y_{2}=\left\vert
Z_{0}\right\vert sign\left( XY_{1}\right).$\\
In all cases, the critical values of our test were calculated  through 50000 replications and the power
of all of the tests were considered from 10000 replications.

To have an idea of the size of the test for random vectors, we have simulated $X, Y \in \mathbb{R}^{5}$ using
$g_1$ and $g_2$ proposed in Subsection 3.3. The power
of the test were $0.051$, $0.048$ and $0.052$ for sample sizes of $30,50$ and $80$, respectively.

\begin{table*}[ht]\caption{ Power comparison for the different test for sample size of $n=30$.}
\label{t4}
\centering
\begin{tabular}{|r|rrrrrrrr|}

    \hline
   Test & HHG  & DCOV & HSIC & N(1,1) & N(1,4) & N(0,4) & N(2,4) & $g_1,g_2$ \\
   \hline 
   Log & 0.594 & 0.154 & 0.610 & 0.710 & 0.759 & 0.321 & \textbf{0.885} & 0.813 \\
   \hline 
   Epsilon & 0.784 & 0.226 & 0.484 & 0.470 & 0.576 & 0.194 & 0.749 & \textbf{0.858} \\
   \hline 
   Quadratic & \textbf{0.687} & 0.302 & 0.530 & 0.197 & 0.155 & 0.170 & 0.147 & 0.144 \\
  
   \hline
   2D-indep & 0.161 & 0.175 & \textbf{0.403} & 0.177 & 0.264 & 0.106 & 0.263 &  0.112 \\

   \hline
   
\end{tabular}             
\end{table*}

\begin{table*}[ht]\caption{ Power comparison for the different test for sample 
size of $n=50$.}
\label{t5}
\centering
\begin{tabular}{|r|rrrrrrrr|}

    \hline
   Test & HHG  & DCOV & HSIC & N(1,1) & N(1,4) & N(0,4) & N(2,4)  & $g_1,g_2$ \\
   \hline 
   Log & 0.936 & 0.386 & 0.958 & 0.998 & 0.999 & \textbf{1.000 }& \textbf{1.000}  & 0.995\\
   \hline
   Epsilon & 0.969 & 0.298 & 0.689 & 0.895 & 0.967 & 0.968 & \textbf{0.999} &  0.984  \\
   \hline
    Quadratic &\textbf{ 0.934} & 0.485 & 0.904 & 0.362  & 0.293 & 0.315 & 0.733 &  0.236 \\
    \hline
      2D-indep & 0.27 & 0.359 & \textbf{0.798} & 0.281 & 0.219 & 0.261 & 0.198 &  0.172\\
    \hline
    
\end{tabular}             
\end{table*}

\begin{table*}[ht]\caption{ Power comparison for the different test for sample size of $n=80$.}
 \label{t6}

\centering
\begin{tabular}{|r|rrrrrrrr|}

    \hline
   Test & HHG  & DCOV & HSIC & N(1,1) & N(1,4) & N(0,4) & N(2,4)  & $g_1,g_2$ \\
   \hline 
   Log & \textbf{1.000} & 0.793  & \textbf{1.000} & \textbf{1.000} & \textbf{1.000} & \textbf{1.000} & \textbf{1.000 }& \textbf{1.000}\\
   \hline
   Epsilon & 0.999 & 0.382 & 0.896 & 0.998 & \textbf{1.000} & \textbf{1.000 }& \textbf{1.000}  &  \textbf{1.000} \\
   \hline
   Quadratic & \textbf{0.996} & 0.725 & 0.971 & 0.595 & 0.545  & 0.535 & 0.480 &  0.416\\
   \hline
  
   2D-indep & 0.544 & 0.751 & \textbf{0.993} & 0.489 & 0.348 & 0.466 & 0.263 &  0.284 \\
   \hline

\end{tabular}             
\end{table*}

\subsection{$X$ and $Y$ are time series}
In this subsection, we consider the case in which $X$ and $Y$ are time series. In all cases
$X$ and $Y$ are time series of length $100$ and the power (due to the computational cost) were calculated by a permutation method for
$m=1.000$ replications (Table~\ref{t7} and Table~\ref{t8})
and $m=100$ replications (Table~\ref{t9}). All the power were calculated using $g_1$ and $g_2$ proposed in Subsection 4.3.
The power for different alternatives and sample sizes in the discrete case are given in Table~\ref{t7}. 
The  AR$(0.1)$ and AR$(0.9)$ means that the time series $X$ is an AR$(1)$ with parameter $0.1$ and $0.9$, respectively.
The case called ARMA$(2,1)$, is an ARMA$(2,1)$ model with parameters $\phi=(0.2,0.5)$ and $\theta =0.2$.   In column 4 of Table~\ref{t7}, $Z$ represents a white noise where $\sigma$ is the
standard deviation of $\sqrt{|X|}$. In Table~\ref{t7} and Table~\ref{t8}, $\varepsilon$ and $\varepsilon '$ are independent white noises  
with $\sigma=1$.
In Table~\ref{t8} are given the power for different alternatives and sample sizes in the continuous case. In this table, $Bm$ represents that $X$ is a Brownian motion
with $\sigma =1$ observed in $[0,1]$ (at times $0,1/100,2/100,...,99/100$) and $fBm$ is a fractional Brownian motion with Hurst parameter $H=0.7$. Finally, Table~\ref{t9} shows the power for cases in which the dependency between $X$ and $Y$ is more difficult to detect. In 
these cases, $Y$ is a fractional Ornstein-Uhlenbeck process driven by a $fBm$ ($X$) for $H=0.5$ and $H=0.7$, which we call $OU$ and $FOU$, respectively. 
A particular linear combination of $FOU$, which we call $FOU(2)$, and whose definition and theoretical developed is found in \cite{Kalemkerian}, is a particular case
of the models proposed in \cite{Arratia}.  
Table~\ref{t9} considers the parameters $\sigma=1, \lambda=0.3$ (column 3) and $\sigma=1, \lambda_{1}=0.3, \lambda_{2}=0.8$ (column 4).
More explicitly, $Y_t=\sigma \int_{-\infty}^{t}e^{-\lambda (t-s)}dX_s$ in column 3 (where $X=\{X_{t}\}$ is a fBm), and 
$Y_t= \dfrac{\lambda_{1}}{\lambda_{1}-\lambda_{2}}\sigma \int_{-\infty}^{t}e^{-\lambda_{1} (t-s)}dX_s+\dfrac{\lambda_{2}}{\lambda_{2}-\lambda_{1}}\sigma \int_{-\infty}^{t}
e^{-\lambda_{2} (t-s)}dX_s $ in column 4 (where $X=\{X_{t}\}$ is a fBm). To give an idea of the size of the test, in column 5 $Y$ is a $Bm$ independent of $X$.

\begin{table*}[ht]\caption{ Power for the case of discrete time series and different sample sizes.}
 \label{t7}

\centering
\begin{tabular}{|r|r|r|r|r|r|}

    \hline
    $n$& $X$ \ \ \ \   & $Y=X^{2} +3\varepsilon$ & $Y=\sqrt{|X|}+Z $ & $Y=\varepsilon X$ & $Y=\varepsilon $ \\
   \hline 
   $30$ & AR$(0,1)$ & 0.350 & 0.214  & 0.772 & 0.051 \\
   \hline
     $50$ & AR$(0,1)$ & 0.592  & 0.402 & 0.962 & 0.050 \\
   \hline
     $100$ & AR$(0,1)$ & 0.999  & 0.698 & 1.000 & 0.046 \\
   \hline
   \hline
   $30$ & AR$(0,9)$ & 1.000 & 0.903 & 1.000 & 0.035 \\
   \hline
     $50$ & AR$(0,9)$ & 1.000 & 0.998  & 1.000 & 0.053 \\
   \hline
     $100$ & AR$(0,9)$ &1.000  & 1.000 & 1.000 & 0.039 \\
   \hline
   \hline
   $30$ & ARMA$(2,1)$ & 0.817 & 0.323 &  0.925 & 0.057 \\
   \hline
     $50$ & ARMA$(2,1)$ & 0.986  & 0.566 & 0.996 & 0.047 \\
   \hline
     $100$ & ARMA$(2,1)$ & 1.000 & 0.921 & 1.000 & 0.051 \\
   \hline

\end{tabular}
\end{table*}

\begin{table*}[ht]\caption{ Power for the case of continuous time series and different sample sizes.}
 \label{t8}

\centering
\begin{tabular}{|r|r|r|r|r|r|}

    \hline
    $n$& $X$ \ \ \ \   & $Y=X^{2}+3\varepsilon$ & $Y=\sqrt{|X|}+\varepsilon $ & $Y=\varepsilon X+3\varepsilon '$ & $Y=\varepsilon$  \\
   \hline 
   $30$ & $Bm$ & 0.770  & 0.519 & 0.402 & 0.060  \\
   \hline
   $50$ & $Bm$ & 0.924 & 0.752 & 0.656 & 0.052 \\
   \hline
   $80$ & $Bm$ &0.994 & 0.923 & 0.839 & 0.040 \\
   \hline
   \hline
    $30$ & $fBm$ & 0.732 & 0.550 & 0.366 & 0.039 \\
   \hline
   $50$ & $fBm$ & 0.883 & 0.805 & 0.586 & 0.040\\
   \hline
   $80$ & $fBm$ & 0.987 & 0.930 & 0.804 & 0.051\\
   \hline
   \hline

\end{tabular}
\end{table*}

\begin{table*}[ht]\caption{ Power where the dependence is between a fractional Brownian motion and its associated FOU and FOU(2), for the cases
  $H=0.5$ ($Bm$) and $H=0.7$ ($fBm$).}
  \label{t9}
	
	\centering
	\begin{tabular}{|r|r|r|r|r|}
		
		\hline
		$n$& $X$ \ \ \ \   & $Y=$FOU & $Y=$FOU$(2)$ & $Y=Bm$ \\
		\hline 
		$30$ & $Bm$ & 0.775 & 0.183  & 0.053 \\
		\hline
		$50$ & $Bm$ & 0.906 & 0.541 & 0.046 \\
		\hline
		$80$ & $Bm$ &0.986 & 0.880 & 0.056 \\
		\hline
		\hline
		$30$ & $fBm$ & 0.380 & 0.106 & 0.045 \\
		\hline
		$50$ & $fBm$ & 0.516 & 0.282 & 0.039 \\
		\hline
		$80$ & $fBm$ & 0.707 & 0.542 & 0.042 \\
		\hline

	\end{tabular}
\end{table*}

\section{Conclusions}

In this work we have presented a new test of independence\ between two
random elements lying in metric spaces. Our test is based on percentages of
recurrences for which we need, for each sample, only the information
obtained by the distance between points. We have obtained the asymptotic distribution of our statistic and we have shown
that the limit distribution under contiguous alternatives has a bias. We have also \ proven the
consistency of the test for a wide class of alternatives, which include the particular case in which $\left( X,Y\right) $
follows a multivariate normal distribution. The performance of the test
measured through the calculation of power through several alternatives has
shown very good results, clearly improving on others in many cases for
different dimensions of the spaces. In future work, we think that the
result can be generalized to the case in which there is some kind of
dependence between the observation of the sample. In addition, the work of the
simulations should be expanded and deepened. 
\section*{Acknowledgments}

Our gratitude to Jos\'e Rafael Le\'on, Ricardo Fraiman, Ernesto Mordecki and Jorge Graneri for their comments that 
were very useful in the preparation of this work. Also the editor of the journal and the two anonymous referees for their 
enriching comments.
\section{Proofs}
\begin{proof}[Proof of Lemma 1]

\[\]Observe that as  $\left( X_{1},Y_{1}\right) ,\left(
X_{2},Y_{2}\right) ,...,\left( X_{n},Y_{n}\right) $ i.i.d, then  \[P\left(
d \left ( X_{i},X_{j} \right )<r\text{ },\text{ }d \left ( Y_i ,Y_j \right ) <s\right) =p_{X,Y\text{ }}(r,s)\] for all $i,j$ such
that $i\neq j.$ Therefore 
\begin{equation*}
\mathbb{E}\left( RR_{n}^{X,Y}(r,s)\right) =\mathbb{E}\left( \frac{1}{n^{2}-n}%
\sum_{i\neq j}\mathbf{1}_{\left\{ d \left ( X_{i},X_{j} \right ) <r%
\text{ },\text{ }d \left ( Y_{i},Y_{j} \right ) <s\right\} }\right) =
\end{equation*}%
\begin{equation*}
\frac{1}{n^{2}-n}\sum_{i\neq j}P\left( d \left ( X_{i},X_{j} \right )
<r\text{ },\text{ }d \left ( Y_{i},Y_{j} \right ) <s\right)
=p_{X,Y}(r,s).
\end{equation*}%
Analogously, $\mathbb{E}\left( RR_{n}^{X}(r)\right) =p_{X}(r)$ and $\mathbb{E}%
\left( RR_{n}^{Y}(s)\right) =p_{Y}(s).$ Given that $X$ and $Y$ are independent, then $$%
\mathbb{E}\left( RR_{n}^{X,Y}(r,s)-RR_{n}^{X}(r)RR_{n}^{Y}(s)\right) =0.$$

Thus, 
\begin{equation*}
\mathbb{COV}\left( E_{n}(r,s) ,  E_{n}(r',s')\right) =
\end{equation*}
\begin{equation*}
 \mathbb{E} \left [
\left(
RR_{n}^{X,Y}(r,s)-RR_{n}^{X}(r)RR_{n}^{Y}(s)\right)\left(
RR_{n}^{X,Y}(r',s')-RR_{n}^{X}(r') RR_{n}^{Y}(s')\right) \right ]=
\end{equation*}%
\begin{equation*}
\mathbb{E}\left[ \left( RR_{n}^{X,Y}(r,s)\right)\left( RR_{n}^{X,Y}(r',s')
\right) \right] -\mathbb{E}\left(
RR_{n}^{X,Y}(r,s)RR_{n}^{X}(r')RR_{n}^{Y}(s')\right) 
\end{equation*}
\begin{equation}\label{varianza}
-\mathbb{E}\left(
RR_{n}^{X,Y}(r',s')RR_{n}^{X}(r)RR_{n}^{Y}(s)\right)+\mathbb{E}\left[ 
RR_{n}^{X}(r)RR_n^{X}(r')\right] \mathbb{E}\left[  RR_{n}^{Y}(s)
RR_n^{Y}(s')%
\right] .
\end{equation}

\bigskip 
\begin{equation*}
\mathbb{E} \left( RR_{n}^{X}(r)RR_n^{X}(r')\right)  =\mathbb{E}\left( 
\frac{1}{n^{2}(n-1)^{2}}\sum_{i\neq j}\sum_{h\neq k}\mathbf{1}_{\left\{
d \left ( X_{i},X_{j} \right ) <r,\text{ }d \left ( X_{h},X_{k} \right )
<r'\right\} }\right) =
\end{equation*}%
\begin{equation}\label{sumas}
\frac{1}{n^{2}(n-1)^{2}}\sum_{i\neq j}\sum_{h\neq k}P\left( d \left ( 
X_{i},X_{j} \right  )<r,\text{ }d \left ( X_{h},X_{k} \right )
<r'\right) .
\end{equation}%
Decomposing (\ref{sumas}) in the terms in which $i,j,k,h$ are
pairwise different, $\left\{ i,j\right\} =\left\{ h,k\right\} $ and $\left\{
i,j,h,k\right\} $ has three elements, and using that the $X-$random vectors
are i.i.d, we obtain that (\ref{sumas}) is equal to   
\begin{equation*}
\frac{n(n-1)(n-2)(n-3)p_{X}(r)p_X (r')+2n(n-1)p_{X}(r)+
	4n(n-1)(n-2)p_{X}^{\left( 3\right) }(r\wedge r')}{n^{2}(n-1)^{2}}
 =
\end{equation*}%
\begin{equation}\label{Rx2}
\frac{n-2}{n(n-1)}\left[ (n-3)p_{X}(r)p_X (r')+4p_{X}^{\left( 3\right) }(r\wedge r')\right]
+o\left( \frac{1}{n}\right) .
\end{equation}%
Analogously 
\begin{equation}\label{Ry2}
\mathbb{E}\left[ \left( RR_{n}^{Y}(s)\right) RR_n^{Y}(s')\right] =
\frac{n-2}{n(n-1)}%
\left( (n-3)p_{Y}(s)p_Y (s')+4p_{Y}^{\left( 3\right) }(s\wedge s')\right) +o\left( \frac{1}{n}%
\right) .
\end{equation}%
Similarly, using that the $\left( X,Y\right) -$random vectors are i.i.d. and
also that $X$ and $Y$ are independent,  
\begin{equation*}
\mathbb{E}\left[  RR_{n}^{X,Y}(r,s)RR_n ^{X,Y}(r',s') \right] =\end{equation*}
 \begin{equation*}                                                     
\frac{(n-2)(n-3)p_{X}(r)p_X (r')p_{Y}(s)p_Y (s')
	+2p_{X}(r)p_{Y}(s)+4(n-2)p_{X}^{\left( 3\right)
	}(r \wedge r')p_{Y}^{\left( 3\right) }(s \wedge s')}{n(n-1)}%
  =
\end{equation*}%

\begin{equation}\label{Rxy2}
\frac{n-2}{n(n-1)}%
\left[ (n-3)p_{X}(r)p_X (r')p_{Y}(s)p_Y(s')+
4p_{X}^{\left( 3\right) }(r \wedge r')p_{Y}^{\left(
3\right) }(s \wedge s')\right] +o\left( \frac{1}{n}\right) .
\end{equation}

With the same technique as in (\ref{Rx2}) and (\ref{Ry2}), we obtain
$\mathbb{E}\left[ RR_{n}^{X,Y}(r,s)RR_{n}^{X}(r')RR_{n}^{Y}(s')%
\right] =$

\begin{equation*}
\mathbb{E}\left( \frac{1}{n^{3}(n-1)^{3}}\sum_{i\neq j}\sum_{h\neq
k}\sum_{l\neq m}\mathbf{1}_{\left\{ d \left ( X_{i},X_{j} \right )
<r,d \left ( Y_{i},Y_{j} \right ) <s,\text{ }d \left ( X_{h},X_{k} \right )
<r',d \left ( Y_{l},Y_{m} \right )
<s'\right\} }\right) =
\end{equation*}%
\begin{equation*}
\frac{1}{n^{3}(n-1)^{3}}\sum_{i\neq j}\sum_{h\neq k}\sum_{l\neq m}P\left(
d \left ( X_{i},X_{j} \right ) <r,d \left ( 
Y_{i},Y_{j} \right )
<s,\text{ }d \left ( X_{h},X_{k} \right ) <r',d \left ( Y_{l},Y_{m} \right )
<s'\right) =
\end{equation*}%
\begin{equation*}
\frac{1}{n^{3}(n-1)^{3}}\sum_{i\neq j}\sum_{h\neq k}\sum_{l\neq m}P\left(
d \left ( X_{i},X_{j} \right ) <r,\text{ }d \left ( X_{h},X_{k} \right )
<r'\right) P\left( d \left ( Y_i ,Y_j \right ) <s,d \left ( Y_{l},Y_{m} \right )
<s' \right) =
\end{equation*}%
\begin{equation*}
\frac{1}{n^{3}(n-1)^{3}}\left[
n(n-1)(n-2)^{2}(n-3)^{2}p_{X}(r)p_X(r')p_{Y}(s)p_Y(s') \right ]+
\end{equation*}
\begin{equation*}
\frac{1}{n^{3}(n-1)^{3}}\left[
4n(n-1)(n-2)^{2}(n-3)p_{X}(r)p_X(r')p_{Y}^{\left( 3\right) }(s \wedge s')%
\right] +
\end{equation*}%
\begin{equation*}
\frac{1}{n^{3}(n-1)^{3}}\left[ 4n(n-1)(n-2)^{2}(n-3)p_{X}^{\left( 3\right)
}(r \wedge r')p_{Y}(s)p_Y(s')
+8n(n-1)(n-2)p_{X}(r)p_{Y}^{\left( 3\right) }(s \wedge s') \right] +
\end{equation*}%
\begin{equation*}
\frac{1}{n^{3}(n-1)^{3}}\left[ 8n(n-1)(n-2)p_{X}^{\left( 3\right)
}(r \wedge r')p_{Y}(s)+2n(n-1)(n-2)(n-3)p_{X}(r)p_X(r')p_{Y}(s)\right] +
\end{equation*}%
\begin{equation*}
\frac{1}{n^{3}(n-1)^{3}}\left[
2n(n-1)(n-2)(n-3)p_{X}(r)p_{Y}(s)p_Y(s')
+16n(n-1)(n-2)^{2}p_{X}^{\left( 3\right)
}(r \wedge r')p_{Y}^{\left( 3\right) }(s\wedge s')\right] +
\end{equation*}%
\begin{equation*}
\frac{1}{n^{3}(n-1)^{3}}4n(n-1)p_{X}(r)p_{Y}(s).
\end{equation*}

Therefore 
\begin{equation*}
\mathbb{E}\left[ RR_{n}^{X,Y}(r,s)RR_{n}^{X}(r')RR_{n}^{Y}(s')\right] =
\end{equation*}
\begin{equation*}
\frac{1}{%
n^{2}(n-1)^{2}}\left[
(n-2)^{2}(n-3)^{2}p_{X}(r)p_X(r')p_{Y}(s)p_Y(s')
+4(n-2)^{2}(n-3)p_{X}(r)p_X(r')p_{Y}^{\left(
3\right) }(s)\right] +
\end{equation*}

\begin{equation*}
\frac{1}{n^{2}(n-1)^{2}}\left[ 4(n-2)^{2}(n-3)p_{X}^{\left( 3\right)
}(r \wedge r' )p_{Y}(s)p_Y(s')+8(n-2)p_{X}(r)p_{Y}^{\left( 3\right) }
(s \wedge s')+8(n-2)p_{X}^{\left( 3\right)
}(r)p_{Y}(s)\right] +
\end{equation*}%
\begin{equation*}
\frac{1}{n^{2}(n-1)^{2}}\left[
2(n-2)(n-3)p_{X}(r)p_X(r')p_{Y}(s)
+2(n-2)(n-3)p_{X}(r)p_{Y}(s)p_Y(s') \right ]+
\end{equation*}
\begin{equation*}
\frac{1}{n^{2}(n-1)^{2}}\left[
16(n-2)^{2}p_{X}^{\left(
3\right) }(r \wedge r')p_{Y}^{\left( 3\right) }(s \wedge s' )
+4p_{X}(r)p_{Y}(s)\right] =
\end{equation*}%
\begin{equation*}\label{triple}
\frac{(n-2)^{2}(n-3)}{n^{2}(n-1)^{2}}\left[ (n-3)p_{X}(r)p_X(r')
p_{Y}(s)p_Y(s')+4\left(
p_{X}^{\left( 3\right) }(r \wedge r')p_{Y}(s)p_Y(s')
+p_{X}(r \wedge r')p_{Y}^{\left( 3\right) }(s)\right) %
\right] \end{equation*}
\begin{equation*}       
+o\left( \frac{1}{n}\right) .
\end{equation*}%
Putting (\ref{Rx2}), (\ref{Ry2}) and (\ref{Rxy2}) in (\ref{varianza}), we
obtain that (\ref{varianza}) is equal to 
\begin{equation*}
\frac{1}{n^{2}(n-1)^{2}}\left[
(n-2)(n-3)(4n-6)p_{X}(r)p_X(r')p_{Y}(s)p_Y(s') \right ]+
\end{equation*}
\begin{equation*}
\frac{1}{n^{2}(n-1)^{2}}\left[
4(n-2)(n^{2}+3n-8)p_{X}^{\left( 3\right)
}(r \wedge r')p_{Y}^{\left( 3\right) }(s \wedge s') \right] +
\end{equation*}%
\begin{equation*}
\frac{-4(n-2)^{2}(n-3)}{n^{2}(n-1)^{2}}\left( p_{X}(r)p_X(r')p_{Y}^{\left(
3\right) }(s \wedge s')+p_{X}^{\left( 3\right) }(r \wedge r')
p_{Y}^{2}(s)\right) +o\left( \frac{1}{n}%
\right) .
\end{equation*}

Then 
\begin{equation*}
\lim_{n\rightarrow +\infty }\mathbb{COV}\left( E_{n}(r,s),E_n(r',s') \right )=
\end{equation*}

\begin{equation*}
4\left( p_{X}^{\left( 3\right)}(r \wedge r')-
p_{X}(r)p_X(r')\right) \left( p_{Y}^{\left( 3\right) }(s \wedge s')-
p_{Y}(s)p_Y(s') \right).
\end{equation*}

\end{proof}

\begin{proof}[Proof of Lemma 2]
\begin{equation*}
\sqrt{n}\left( RR_{n}^{X,Y}(r,s)-RR_{n}^{X}(r)RR_{n}^{Y}(s)\right) =
\end{equation*}
\begin{equation*}
\frac{%
\sqrt{n}}{n(n-1)}\sum_{(i,j)\in I_{2}^{n}}\mathbf{1}_{\left\{ d \left ( 
X_{i},X_{j} \right  )<r,\text{ }d \left ( Y_{i},Y_{j} \right )
<s\right\} }-\sqrt{n}RR_{n}^{X}(r)RR_{n}^{Y}(s)=
\end{equation*}%
\begin{equation*}
\frac{\sqrt{n}}{n(n-1)(n-2)(n-3)}\sum_{(i,j,h,k)\in I_{4}^{n}}\mathbf{1}%
_{\left\{ d \left ( X_{i},X_{j} \right ) <r,\text{ }d \left ( Y_i ,Y_j \right ) <s\right\} }-\sqrt{n}RR_{n}^{X}(r)RR_{n}^{Y}(s)=
\end{equation*}%
\begin{equation*}
E_{n}^{\prime }(r,s)-H_{n}(r,s)
\end{equation*}%
where $H_{n}(r,s)=$

\begin{equation*}
\sqrt{n}\left( RR_{n}^{X}(r)RR_{n}^{Y}(s)-\frac{1}{%
n(n-1)(n-2)(n-3)}\sum_{(i,j,k,h)\in I_{4}^{n}}\mathbf{1}_{\left\{ d \left ( 
X_{i},X_{j} \right  )<r,\text{ }d \left ( Y_{h},Y_{k} \right )
<s\right\} }\right) .
\end{equation*}%
Then, $H_{n}(r,s)$ is equal to 
\begin{equation*}
\frac{\sqrt{n}}{n^{2}(n-1)^{2}} \sum_{(i,j)\in I_{2}^{n}}%
\mathbf{1}_{\left\{ d \left ( X_{i},X_{j} \right ) <r\right\}
}\sum_{(h,k)\in I_{2}^{n}}\mathbf{1}_{\left\{ d \left ( Y_{h},Y_{k} \right )
<s\right\} }
\end{equation*}
\begin{equation*}
-\frac{\sqrt{n}}{n(n-1)(n-2)(n-3)}\sum_{(i,j,k,h)\in
I_{4}^{n}}\mathbf{1}_{\left\{ d \left ( X_{i},X_{j} \right ) <r,\text{ }%
d \left ( Y_{h},Y_{k} \right ) <s\right\} } =
\end{equation*}%
\begin{equation*}
\frac{\sqrt{n}}{n^{2}(n-1)^{2}} \frac{1}{n(n-1)}\sum_{(i,j)\in
I_{2}^{n}}\sum_{(h,k)\in I_{2}^{n}}\mathbf{1}_{\left\{ d \left ( 
X_{i},X_{j} \right  )<r,\text{ }d \left ( Y_{h},Y_{k} \right )
<s\right\} }
\end{equation*}
\begin{equation}\label{Hnn}
-\frac{\sqrt{n}}{n(n-1)(n-2)(n-3)}\sum_{(i,j,k,h)\in I_{4}^{n}}\mathbf{1}%
_{\left\{ d \left ( X_{i},X_{j} \right ) <r,\text{ }d \left ( Y_{h},Y_{k} \right )
<s\right\} } .
\end{equation}

Now, we decompose 
\begin{equation*}
\sum_{(i,j)\in I_{2}^{n}}\sum_{(h,k)\in I_{2}^{n}}\mathbf{1}_{\left\{
d \left ( X_{i},X_{j} \right ) <r,\text{ }d \left ( Y_{h},Y_{k} \right ) <s\right\} }=
\end{equation*}
\begin{equation*}
\sum_{(i,j,k,h)\in I_{4}^{n}}\mathbf{1}%
_{\left\{ d \left ( X_{i},X_{j} \right ) <r,\text{ }d \left ( Y_{h},Y_{k} \right )
<s\right\} }+4\sum_{(i,j,k)\in I_{3}^{n}}\mathbf{1}%
_{\left\{ d \left ( X_{i},X_{j} \right )<r,\text{ }d \left ( Y_{i},Y_{k} \right )
<s\right\} }+
\end{equation*}
\begin{equation*}
2\sum_{(i,j)\in I_{2}^{n}}\mathbf{1}%
_{\left\{ d \left ( X_{i},X_{j} \right )<r,\text{ }d \left ( Y_i ,Y_j \right ) <s\right\} }
\end{equation*}

\ and substituting in (\ref{Hnn}) we obtain that (\ref{Hnn}) is equal to 
\begin{equation*}
\frac{\sqrt{n}}{n(n-1)}\left( \left( \frac{1}{n(n-1)}-\frac{1}{(n-2)(n-3)}%
\right) \sum_{(i,j,k,h)\in I_{4}^{n}}\mathbf{1}_{\left\{ d \left ( 
X_{i},X_{j} \right  )<r,\text{ }d \left ( Y_{h},Y_{k} \right )
<s\right\} }\right) +
\end{equation*}%
\begin{equation*}
\frac{\sqrt{n}}{n^{2}(n-1)^{2}}\left( 4\sum_{(i,j,k)\in I_{3}^{n}}\mathbf{1}%
_{\left\{ d \left ( X_{i},X_{j} \right ) <r,\text{ }d \left ( Y_{i},Y_{k} \right )
<s\right\} }+2\sum_{(i,j)\in I_{2}^{n}}\mathbf{1}%
_{\left\{ d \left ( X_{i},X_{j} \right ) <r,\text{ }d \left ( Y_i ,Y_j \right ) <s\right\} }\right) =
\end{equation*}

\begin{equation*}
\frac{\sqrt{n}}{n^{2}(n-1)^{2}(n-2)(n-3)}\sum_{(i,j,k,h)\in I_{4}^{n}}%
\mathbf{1}_{\left\{ d \left ( X_{i},X_{j} \right ) <r,\text{ }d \left ( Y_{h},Y_{k} \right )
<s\right\} }+
\end{equation*}%
\begin{equation}\label{acoto}
\frac{\sqrt{n}}{n^{2}(n-1)^{2}}\left( 4\sum_{(i,j,k)\in I_{3}^{n}}\mathbf{1}%
_{\left\{ d \left ( X_{i},X_{j} \right ) <r,\text{ }d \left ( Y_{i},Y_{k} \right )
<s\right\} }+2\sum_{(i,j)\in I_{2}^{n}}\mathbf{1}%
_{\left\{ d \left ( X_{i},X_{j} \right ) <r,\text{ }d \left ( Y_i ,Y_j \right ) <s\right\} }\right)
\end{equation}%
Observe that (\ref{acoto}) it is bounded between $0$ and 
\[\frac{\sqrt{n}}{%
n^{2}(n-1)^{2}}\left( 4n(n-1)(n-2)+2n(n-1)\right) =\allowbreak \frac{1}{%
\sqrt{n}}\frac{4n-6}{n-1}<\frac{4}{\sqrt{n}}.\]

\end{proof}

\begin{proof}[Proof of Theorem \ref{asymptotic under H0}]
\[\]
Every  continuous function $h:\mathbb{R}\rightarrow \mathbb{R}$ with finit limits as $x\rightarrow \pm \infty$ is uniformly continuous.
Therefore given $\varepsilon >0$, exist $\delta>0$  such that $|F(x)-F(y)|\leq \varepsilon^{2} /8$ and $|G(x)-G(y)|\leq \varepsilon^{2} /8$ for all $(x,y)$ 
such that $|x-y|<\delta$, where $F$ and $G$ are the distribution functions
of $d(X_1,X_2)$ and $d(Y_1,Y_2)$ respectively.
If $H_0$ is true,  consider for each $r,s>0$ the functions 
 $f_{r,s}:(S_X \times S_Y )^{4}\rightarrow \mathbb{R}$ defined by 
 \[ f_{r,s}\left( x,y,x^{\prime },y^{\prime },x^{\prime \prime },y^{\prime
\prime },x^{\prime \prime \prime },y^{\prime \prime \prime }\right) =\mathbf{%
1}_{\left\{ d \left ( x,x^{\prime }\right ) <r,\text{ } d\left ( 
y,y^{\prime }\right ) <s\right\} }-\mathbf{1}_{\left\{ d\left (
x,x^{\prime }\right ) <r,\text{ } d\left ( y^{\prime \prime },y^{\prime
\prime \prime }\right ) <s\right\} }\]where $x,x^{\prime },x^{\prime \prime },x^{\prime \prime \prime }\in S_X$ and $y,y^{\prime },y^{\prime \prime },y^{\prime \prime \prime }\in S_Y.$
and consider the family $\mathcal{F}=\{f_{r,s}\}_{r,s>0}$.
To simplify the notation,  we call $z=\left(
x,y,x^{\prime },y^{\prime },x^{\prime \prime },y^{\prime \prime },x^{\prime
\prime \prime },y^{\prime \prime \prime }\right) $ throughout the demonstration.

Observe that 
\begin{equation*}
E_{n}^{\prime }(r,s)=\frac{\sqrt{n}}{n(n-1)(n-2)(n-3)}\sum_{(i,j,k,h)\in
I_{4}^{n}}f_{r,s}\left(
X_{i},Y_{i},X_{j},Y_{j},X_{h},Y_{h},X_{k},Y_{k}\right)
\end{equation*}%
then the process $\left\{ E_{n}^{\prime }(r,s)\right\} _{r,s>0}$ is an $U-$%
process of order $4.$

To obtain the convergence, according to Arcones \& Gin\'{e}'s Theorem 4.10,
it is enough to prove that $$\int_{0}^{+\infty }\left( \log N_{\left[ { \ }%
\right] }^{\left( 2\right) }\left( \varepsilon ,\mathcal{F},P^{4}\right) \right)
^{1/2}d\varepsilon <+\infty .$$

If $\varepsilon \geq 2,$ then $-1\leq f_{r,s}(z)\leq 1$ for all $z\in 
(S_X \times S_Y)^{4}$ and $r,s>0.$ Then $\mathcal{L}=\left\{ -1\right\} $, 
$\mathcal{U}=\left\{
1\right\} $ satisfied (\ref{entropia}) Thus, $N_{\left[ {\ }\right] }^{\left( 2\right)
}\left( \varepsilon ,\mathcal{F},P^{4}\right) =1$, therefore $\int_{0}^{+\infty
}\left( \log N_{\left[ {\ }\right] }^{\left( 2\right) }\left( \varepsilon
,\mathcal{F},P^{4}\right) \right) ^{1/2}d\varepsilon =\int_{0}^{2}\left( \log N_{\left[
{\ }\right] }^{\left( 2\right) }\left( \varepsilon ,\mathcal{F},P^{4}\right) \right)
^{1/2}d\varepsilon .$

If $\varepsilon <2$,  we take $T>0$ such that $\max \left\{ 1-F(T) , 1-G(T) \right\} <\varepsilon^{2}/8 $, then we partition $%
\left[ 0,+\infty \right) $ into $m+1$ subintervals of the form $\left[ \frac{iT}{m},%
\frac{(i+1)T}{m}\right) $ such that $\frac{T}{m}<\delta, $ where $\frac{(m+1)T}{m}$ is interpreted
as $+\infty .$ Define the following functions 
\begin{equation*}
g_{i,j}(z)=\left\{ 
\begin{array}{ccc}
\mathbf{1}_{\left\{ d \left ( x,x^{\prime }\right ) <\frac{iT}{m},\text{ 
}d\left ( y,y^{\prime }\right ) <\frac{jT}{m}\right\} } & \text{for} & 
i,j=1,2,...,m \\ 
0 & \text{for} & i=0\text{ or }j=0%
\end{array}%
\right .
\end{equation*}

and

\begin{equation*}
h_{i,j}(z)=\left\{ 
\begin{array}{ccc}
\mathbf{1}_{\left\{ d\left ( x,x^{\prime }\right ) <\frac{iT}{m},\text{ 
}d\left ( y^{\prime \prime},y^{\prime \prime \prime }\right ) <\frac{jT}{m}\right\} } & \text{for } & 
i,j=1,2,...,m \\ 
\mathbf{1}_{\left\{ d \left ( x,x^{\prime }\right ) <\frac{iT}{m}%
\right\} } & \text{for} & i=1,2,...,m,\text{ }j=m+1 \\ 
\mathbf{1}_{\left\{ d \left ( y^{\prime \prime},y^{\prime \prime \prime }\right ) <\frac{jT}{m}%
\right\} } & \text{for } & j=1,2,...,m,\text{ }i=m+1 \\ 
1 & \text{for} & i=j=m+1%
\end{array}%
\right . .
\end{equation*}%
Observe that for each $r,s>0$ there exists $i,j\in \left\{
0,1,2,...,m\right\} $ such that $\frac{iT}{m}\leq r<\frac{(i+1)T}{m}$ and  $%
\frac{jT}{m}\leq s<\frac{(j+1)T}{m}.$ 

Then 
\begin{equation*}
g_{i,j}(z)-h_{i+1,j+1}(z)\leq f_{r,s}(z)\leq g_{i+1,j+1}(z)-h_{i,j}(z)\text{
for all }z\in (S_X \times S_Y)^{4},
\end{equation*}%
Thus $\mathcal{L}=\left\{ l_{i,j}\right\} $ and $\mathcal{U}=\left\{ u_{i,j}\right\} $ where $%
l_{i,j}(z)=g_{i,j}(z)-h_{i+1,j+1}(z)$ and $%
u_{i,j}(z)=g_{i+1,j+1}(z)-h_{i,j}(z)$ for $i,j=0,1,2,...,m.$ Also 
\begin{equation}\label{Egij}
\mathbb{E}\left( u_{i,j}(Z)-l_{i,j}(Z)\right) ^{2}\leq 2\left( \mathbb{E}%
\left( g_{i+1,j+1}(Z)-g_{i,j}(Z)\right) ^{2}+\mathbb{E}\left(
h_{i+1,j+1}(Z)-h_{i,j}(Z)\right) ^{2}\right) .
\end{equation}%

Define the sets $A_{i,j}:=\left[ 0,\frac{(i+1)T}{m}\right) \times \left[ 0,%
\frac{(j+1)T}{m}\right) -\left[ 0,\frac{iT}{m}\right) \times \left[ 0,\frac{%
jT}{m}\right) $, then 
\begin{equation*}
\mathbb{E}\left( g_{i+1,j+1}(Z)-g_{i,j}(Z)\right) ^{2}=\mathbb{E}\left( 
\mathbf{1}_{A_{i,j}}(Z)\right)
\end{equation*}
\begin{equation*}
\leq P\left( \frac{iT}{m}\leq d \left ( 
X_{1},X_{2} \right  )<\frac{(i+1)T}{m}\right) +P\left( \frac{jT}{m}\leq
d \left ( Y_{1},Y_{2}\right ) <\frac{(j+1)T}{m}\right) \leq
\end{equation*}

\bigskip 
\begin{equation}\label{norma_u-v}
F\left( \frac{(i+1)T}{m}\right) -F\left( \frac{iT}{m}\right) +G\left( \frac{%
(j+1)T}{m}\right) -G\left( \frac{jT}{m}\right) \leq \varepsilon^{2}/4.
\end{equation}%

Analogously, 
\begin{equation}\label{Ehij}
\mathbb{E}\left( h_{i+1,j+1}(Z)-h_{i,j}(Z)\right) ^{2}\leq \varepsilon^{2}/4.
\end{equation}%

\bigskip putting (\ref{Ehij}) and (\ref{norma_u-v}) in (\ref{Egij}) we obtain that $\mathbb{E}\left(
u_{i,j}(Z)-l_{i,j}(Z)\right) ^{2}\leq \varepsilon ^{2}.$

Lastly, observe that the cardinal of $\mathcal{L}$ and $\mathcal{U}$ is $\left( m+1\right) ^{2}$%
, then 
\begin{equation*}
N_{\left[ {\ }\right] }^{\left( 2\right) }\left( \varepsilon ,\mathcal{F},P^{4}\right)
\leq \frac{cte}{\varepsilon ^{4}}\text{, thus }\int_{0}^{2}\left( \log N_{%
\left[ {\ }\right] }^{\left( 2\right) }\left( \varepsilon ,\mathcal{F},P^{4}\right)
\right) ^{1/2}d\varepsilon <+\infty .
\end{equation*}

\end{proof}

\begin{proof}[Proof of Theorem \ref{consistent theorem}]
\[ \]
Define $\mu \left( r,s\right) =P \left ( d( X_{1},X_{2}) <r,%
\text{ } d( Y_{1},Y_{2}) <s\right) -P\left(  d( X_{1},X_{2})
<r\right) P\left(  d( Y_{1},Y_{2})
<s\right) .$
Then,  $r_{0},s_{0}>0$ exist, such that $\mu ^{2}\left( r_{0},s_{0}\right) >0$,
thus  $\varepsilon >0$ exist and $A\subset \left[ 0,+\infty \right) ^{2}$
such that $\left( r_{0},s_{0}\right) \in A$ and $\mu ^{2}\left( r,s\right)
>\varepsilon $ for all $\left( r,s\right) \in A.$ Then, as $n\rightarrow +\infty,$ 
\[n\int_{0}^{+\infty
}\int_{0}^{+\infty }\mu ^{2}\left( r,s\right) g(r,s)drds\geq n\varepsilon
\iint_{A}g(r,s)drds\rightarrow +\infty. \] 
Now, using that $\left( a+b\right) ^{2}\leq 2\left( a^{2}+b^{2}\right) $ we
obtain that $n\int_{0}^{+\infty }\int_{0}^{+\infty }\mu ^{2}\left(
r,s\right) g(r,s)drds\leq $ 
\begin{equation*}
2n\int_{0}^{+\infty }\int_{0}^{+\infty }\left(
RR_{n}^{X,Y}(r,s)-RR_{n}^{X}(r)RR_{n}^{Y}(s)-\mu \left( r,s\right) \right)
^{2}g(r,s)drds+
\end{equation*}
\begin{equation*}
2n\int_{0}^{+\infty }\int_{0}^{+\infty }\left(
RR_{n}^{X,Y}(r,s)-RR_{n}^{X}(r)RR_{n}^{Y}(s)\right) ^{2}g(r,s)drds
\end{equation*}
Thus 
\begin{equation*}
T_{n}=n\int_{0}^{+\infty }\int_{0}^{+\infty }\left(
RR_{n}^{X,Y}(r,s)-RR_{n}^{X}(r)RR_{n}^{Y}(s)\right) ^{2}g(r,s)drds\overset{P}%
{\rightarrow }+\infty \text{ as }n\rightarrow +\infty .
\end{equation*}
\end{proof}

\begin{proof}[Proof of Corollary \ref{normal consistency}]
	\[\]
	 Because all of the norms in $\mathbb{R}^{p}$ and $\mathbb{R}^{q}$ are
	equivalent, it is enough to give the proof for the Euclidean norm case. We use
	that if $\left( Z,T\right) $ has centered normal bivariate distribution,
	then $\mathbb{COV}\left( Z^{2},T^{2}\right) =2\left( \mathbb{COV}\left(
	Z,T\right) \right) ^{2}.$
	
	Let us call  $X=\left( X_{\left( 1\right) },X_{\left( 2\right)
	},...,X_{\left( p\right) }\right) $ and $Y=\left( Y_{\left( 1\right)
	},Y_{\left( 2\right) },...,Y_{\left( q\right) }\right) .$ Then 
	\begin{equation*}
	\mathbb{COV}\left( \left\Vert X\right\Vert ^{2},\left\Vert Y\right\Vert
	^{2}\right) =\mathbb{COV}\left( \sum_{i=1}^{p}X_{\left( i\right)
	}^{2},\sum_{j=1}^{q}Y_{\left( j\right) }^{2}\right)
	=2\sum_{i=1}^{p}\sum_{j=1}^{q}\left( \mathbb{COV}\left( X_{\left( i\right)
	},Y_{\left( j\right) }\right) \right) ^{2}.
	\end{equation*}
	
	If $X$ and $Y$ are not independent, then $i$ and $j$ exist such that $%
	\mathbb{COV}\left( X_{\left( i\right) },Y_{\left( j\right) }\right)
	\neq 0,$ then $\mathbb{COV}\left( \left\Vert X\right\Vert ^{2},\left\Vert
	Y\right\Vert ^{2}\right) >0,$ then $\left\Vert X\right\Vert ^{2}$ and $%
	\left\Vert Y\right\Vert ^{2}$ are not independent, therefore $\left\Vert
	X\right\Vert $ and $\left\Vert Y\right\Vert $ are not independent, and then
	exist $r$ and $s$ positive numbers such that $P\left( \left\Vert
	X\right\Vert <r,\text{ }\left\Vert Y\right\Vert <s\right) \neq
	P\left( \left\Vert X\right\Vert <r\right) P\left( \left\Vert
	Y\right\Vert <s\right) .$ If we apply this argument for $X_{1}-X_{2}$
	and $Y_{1}-Y_{2}$ instead $X$ and $Y$, then we obtain that 
	\begin{equation*}
	P\left( \left\Vert X_{1}-X_{2}\right\Vert <r,\text{ }d \left\Vert Y_{1}-Y_{2}\right\Vert <s\right) \neq P\left( \left\Vert X_{1}-X_{2}\right\Vert<r\right) P\left( \left\Vert Y_{1}-Y_{2}\right\Vert <s\right) .
	\end{equation*}
Lastly, the result follows from Theorem 2.
	\end{proof}

\begin{proof}[Proof of Proposition 1]
 \begin{equation*}
\mathbb{E}^{\left( n\right) }\left( RR_{n}^{X,Y}(r,s)\right) =\mathbb{E}%
^{\left( n\right) }\left( \frac{1}{N}\sum_{\left( i,j\right) \in I_{2}}%
\mathbf{1}_{\left\{ d( X_{i},X_{j}) <r,\text{ } d( Y_{i},Y_{j}) <s\right\} }\right) =
\end{equation*}
\begin{equation}\label{pn}
P^{\left( n\right) }\left(
d( X_{i},X_{j}) <r,\text{ } d( Y_{i},Y_{j}) <s\right) .
\end{equation}%
 Define $A_{r,s}:=\left\{ \left( x_{1},y_{1},x_{2},y_{2}\right) \in 
\mathbb{R}^{2p+2q}:\text{ } d( x_{1},x_{2})  <r,\text{ }%
d( y_{1},y_{2}) <s\right\} ,$ then (\ref{pn}) is equal to%
\begin{equation*}
c_{n}^{2}\left( \delta \right)
\iiiint_{A_{r,s}}f_{X,Y}^{(n)}(x_{1},y_{1})f_{X,Y}^{(n)}(x_{2},y_{2})dx_{1}dx_{2}dy_{1}dy_{2}=
\end{equation*}%
\begin{equation*}
c_{n}^{2}\left( \delta \right)
\iiiint_{A_{r,s}}f_{X}(x_{1})f_{Y}(y_{1})f_{X}(x_{2})f_{Y}(y_{2}) \times \end{equation*}
\begin{equation*}
\left( 1+%
\frac{\delta }{2\sqrt{n}}k_{n}(x_{1},y_{1})\right) ^{2}\left( 1+\frac{\delta 
}{2\sqrt{n}}k_{n}(x_{2},y_{2})\right) ^{2}dx_{1}dx_{2}dy_{1}dy_{2}=
\end{equation*}%
\begin{equation*}
c_{n}^{2}\left( \delta \right) p_{X}^{\left( 0\right)
}(r)p_{Y}^{\left( 0\right) }(s)+c_{n}^{2}\frac{\delta }{\sqrt{n}} \times 
\end{equation*}
\begin{equation*}
\iiiint_{A_{r,s}}%
\left( k_{n}(x_{1},y_{1})+k_{n}(x_{2},y_{2})\right)
f_{X}(x_{1})f_{Y}(y_{1})f_{X}(x_{2})f_{Y}(y_{2})dx_{1}dx_{2}dy_{1}dy_{2}%
\end{equation*}%
\[+\varepsilon _{n}\left( r,s\right)  \]
where $\left\vert \varepsilon _{n}\left( r,s\right) \right\vert \leq \frac{c%
}{\sqrt{n}}$ for all $r,s>0$ and $c$ is a constant.%
\begin{equation*}
\mathbb{E}^{\left( n\right) }\left( RR_{n}^{X}(r)RR_{n}^{Y}(s)\right) =\frac{%
1}{N^{2}}\mathbb{E}^{\left( n\right) }\left( \sum_{\left( i,j\right) \in
I_{2},\text{ }\left( h,k\right) \in I_{2}}\mathbf{1}_{\left\{ d( X_{i},X_{j}) <r,\text{ } d(Y_{h},Y_{k})
<s\right\} }\right) =
\end{equation*}%
\begin{equation*}
c_{n}^{2}\left( \delta \right) \frac{\left( n-2\right) \left( n-3\right) }{N}%
p_{X}^{\left( 0\right) }(r)p_{Y}^{\left( 0\right) }(s)+\frac{2}{N}P^{\left(
n\right) }\left( A_{r,s}\right) +
\end{equation*}
\begin{equation*}
\frac{4(n-2)}{N}P^{\left( n\right) }\left(
d(X_{1},X_{2}) <r,\text{ } d ( Y_{1},Y_{3}) <s\right) .
\end{equation*}

Therefore 
\begin{equation*}
\mathbb{E}^{\left( n\right) }\left( E_{n}(r,s)\right) =\sqrt{n}\mathbb{E}%
^{\left( n\right) }\left(
RR_{n}^{X}(r)RR_{n}^{Y}(s)-RR_{n}^{X}(r)RR_{n}^{Y}(s)\right) =
\end{equation*}%
\begin{equation*}
\sqrt{n}c_{n}^{2}\left( \delta \right)  \frac{4n-6}{N}p_{X}^{\left(
0\right) }(r)p_{Y}^{\left( 0\right) }(s)+\delta c_{n}^{2}\left( \delta \right) \times
\end{equation*}
\begin{equation*}
\iiiint_{A_{r,s}}\left( k_{n}(x_{1},y_{1})+k_{n}(x_{2},y_{2})\right)
f_{X}(x_{1})f_{Y}(y_{1})f_{X}(x_{2})f_{Y}(y_{2})dx_{1}dx_{2}dy_{1}dy_{2}+%
\varepsilon _{n}\left( r,s\right) .
\end{equation*}

Then $\mathbb{E}^{\left( n\right) }\left( E_{n}(r,s)\right) \rightarrow$
\begin{equation*}
 \delta
\iiiint_{A_{r,s}}\left( k(x_{1},y_{1})+k(x_{2},y_{2})\right)
f_{X}(x_{1})f_{Y}(y_{1})f_{X}(x_{2})f_{Y}(y_{2})dx_{1}dx_{2}dy_{1}dy_{2}%
\end{equation*}
as $n\rightarrow +\infty.$
\end{proof}


\begin{thebibliography}{}
 
 \bibitem{Arcones} Arcones, M. A. and  Gin\'e, E. {(1993)}. {Limit Theorems for U-Processes.} {The Annals of Probability 
 Vol 21-3, 1494-1542.}     
 
\bibitem{Arratia} {Arratia, A., Caba\~{n}a, A. \& Caba\~{n}a, E.,} (2016). A construction of Continuous
time ARMA models by iterations of Ornstein-Uhlenbeck process, \textit{SORT} Vol 40 (2) 267-302. 
\bibitem{Bakirov} {Bakirov, N. K., Rizzo, M. L., and Székely, G. J.} {(2006)}. {A multivariate non-
parametric test of independence}. {Journal of Multivariate Analysis, 97(8):1742-1756}.


  \bibitem{Beran} Beran, R., Bilodeau, M., and Lafaye de Micheaux, P. {(2007)}. {Nonparametric
tests of independence between random vectors.} {Journal of Multivariate Analysis.
98(9):1805–1824.}
\bibitem{Bilodeau}  Bilodeau, M. and Lafaye de Micheaux, P. {(2005)}. {A multivariate empirical
caracteristic function test of independence with normal marginals}. {Journal of
ultivariate Analysis, 95(2):345–369}.

\bibitem{Blomqvist} Blomqvist, N. {(1950)}. On a measure of dependence between two random variables. {
The Annals of Mathematical Statistics 593-600.} 
\bibitem{Boglioni} Boglioni, G. {(2016)}. {A consistent test of independence between random vectors.}
{https://papyrus.bib.umontreal.ca/xmlui/bitstream/handle/1866/18773/Bo \\ glioni\_Beaulieu\_Guillaume\_
2016\_memoire.pdf?sequence=2.} 

\bibitem{henry} Caba\~{n}a, E. M. {(1997)}. {Contiguidad, pruebas de ajuste y Procesos Emp\'i ricos
Transformados.} 
{D\'ecima escuela venezolana de Matem\'aticas.} 
\bibitem{Eckman} Eckmann, J. P, Oliffson Kamphorst, S, Ruelle, D. {(1987)}. {Recurrence plots of dynamical 
systems.} {Europhys. Lett. 4, 973-977.}  
\bibitem{Galton} Galton, F.  {(1888)}. { Co-relations and their measurement, chiefly from 
anthropometric data.} {Proceedings of the Royal Society of London.
45(273-279):135–145.}   
\bibitem{Gine-Zinn} Gin\'e, E. and Zinn, J. {(1986)}.  {Lectures on the central limit theorem for empirical processes.} 
Lecture Note in Math. 1221. 50-113. Springer, New York.
\bibitem{Gretton} Gretton, A., Bousquet, O., Smola, A. and Schölkopf, B. {(2005)}.
{Measuring statistical dependence with 
hilbert-scmidts norms.} {International Conference on algorithmic learning theory. 63-77. 
Springer.}   
\bibitem{Heller}  Heller, R., Heller, Y., and Gorfine, M. {(2012)}. {A consistent multivariate test
of association based on ranks of distances.} {Biometrika. 100(2):503–510.}  
\bibitem{Hoeffding} Hoeffding, W. {(1961)}. {The strong law of large numbers for $U$-statistics.} http://www.lib.ncsu.edu/resolver/1840.4/2128
 \bibitem{Hoeffding2} Hoefdding, W. {(1948)}. {A non-parametric test of independence.} {The Annals of Mathematical Statistics. 546-557.}
 \bibitem{Genest} Genest, G. and R\'emillard, B. {(2004)}. {Test of independence and randomness based on the empirical copula process.}
 {Test, 13(2): 335-369.}
 \bibitem{Kalemkerian} Kalemkerian, J. {(2017)}. {Fractional Iterated Ornstein-Uhlenbeck Processes.}
 {arXiv preprint arXiv:1709.07143v1[math.ST].} 
 \bibitem{Kojadinovic} Kojadinovic, I. and Holmes, M. {(2009)}. {Tests of independence among continuous random vectors based on
 cram\'er-von mises functional of the empirical copula process.} {Journal of Multivariate Analysis, 100(6): 1137-1154.}
 \bibitem{Kendall} Kendall, M. G. {(1938)}. {A new measure of rank correlation.} {Biometrika,
30(1/2):81–93. }    
\bibitem{Le Cam} Le Cam, L. \& Yang, G. L. {(1990)}. {Asymptotics in Statistics. Some Basic Concepts.}
{Springer, New York.}  
 \bibitem{Marwan} Marwan, N. {(2008)}. { A Historical Review of Recurrence Plots. }
 {European Physical Journal—Special Topics, 164, 3-12.}       
 \bibitem{Oosterhoff} Oosterhoof, J. \& Van Zwet, W. R. {(1979)}. { A note on contiguity and Hellinger
 distance.} { Contributions to Statistics. Jaroslav H\'ajek Memorial Volume (J. Jorecová, ed)
 Reidel Dordrecht.  157-166}     
 \bibitem{Pearson} Pearson, K. {(1898)}. {Mathematical contributions to the theory of evolution
on a form of spurious correlation which may arise when indices are used in
the measurement of organs.} { Proceedings of the Royal Society of London.
60(359-367):489–498. }   
\bibitem{Spearman} Spearman, C. {(1904)}. { The proof and measurement of association between
two things.} {The American journal of psychology. 15(1):72–101.} 
\bibitem{Szekely} Sz\'ekely, G. J., Rizzo, M. L., Bakirov, N. K., et al. {(2007)}. {  Measuring and testing 
dependence by correlation of distances.} { The Annals of Statistics. 35(6):2769–
2794.} 
\bibitem{Szekely2} Sz\'ekely, G. J., Rizzo, M. L., et al. {(2009)}. { Brownian distance covariance.} { The
annals of applied statistics, 3(4):1236–1265. }   


 \bibitem{Wilks} Wilks, S. {(1935)}. { On the independence of $k$ sets of normality distributed
 statistical variables. } {Econometrica, Journal of the Econometric Society. 309-326.} 
 
\end{thebibliography}

\end{document}